\documentclass[12pt,reqno]{amsart}
\usepackage{amsmath, amsfonts, amssymb}
\usepackage{latexsym}
\usepackage{enumerate}
\usepackage{cite}
\usepackage{bm}
\usepackage{mathtools}
\usepackage{eucal}
\usepackage{mathrsfs}
\usepackage[usenames]{color}
\usepackage{xcolor}
\definecolor{refkey}{gray}{.45}
\definecolor{labelkey}{gray}{.45}
\usepackage[]{hyperref}

\numberwithin{equation}{section}

  \newcommand{\R}{\mathbb{R}} 
  
  \newcommand{\e}{\varepsilon} 
  \newcommand{\h}{H^{\perp}} 
  
  \newcommand{\s}{\mathbb{S}}

  \newtheorem{lemma}{Lemma}[section]
  \newtheorem{theorem}[lemma]{Theorem}

  \newtheorem{defi}[lemma]{Definition}
  \newtheorem{coro}[lemma]{Corollary}
  
  \newtheorem{prop}[lemma]{Proposition}
  \newtheorem{remark}[lemma]{Remark}
  \newcommand{\blb}{\raise.3ex\hbox{$\scriptstyle \pmb \lbrack$}}
  \newcommand{\sblb}{\raise.1ex\hbox{$\scriptscriptstyle \pmb \lbrack$}}
  \newcommand{\brb}{\raise.3ex\hbox{$\scriptstyle \pmb \rbrack$}}
  \newcommand{\sbrb}{\raise.1ex\hbox{$\scriptscriptstyle \pmb \rbrack$}}
  \newcommand{\bla}{\raise.2ex\hbox{$\scriptstyle\pmb \langle$}}
  \newcommand{\sbla}{\raise.1ex\hbox{$\scriptscriptstyle\pmb \langle$}}
  \newcommand{\bra}{\raise.2ex\hbox{$\scriptstyle\pmb \rangle$}}
  \newcommand{\sbra}{\raise.1ex\hbox{$\scriptscriptstyle\pmb \rangle$}}
  \newcommand{\blrb}{\raise.3ex\hbox{$\scriptstyle \pmb | $}}
  \newcommand{\sblrb}{\raise.1ex\hbox{$\scriptscriptstyle \pmb | $}}

  \newcommand{\Q}{\mathbb{Q}} 
  \newcommand{\avg}{M_{p,\alpha}^{(t,\lambda)}}
  \newcommand{\avgb}{M_{p,\beta}^{(t,\lambda)}}
  \newcommand{\avgc}{M_{p,\gamma}^{(t,\lambda)}}
  
  \newcommand{\lpcurv}{(1-t) \times_{p,\alpha} A +_{p,\alpha} t \times_{p,\alpha} B} 
  \newcommand{\lpmin}{(1-t) \times^{p,\alpha} A +^{p,\alpha} t \times^{p,\alpha} B}

    \author[Michael Roysdon]{Michael Roysdon}
  \address[Michael Roysdon]{
  	School of Mathematical Sciences, Tel Aviv University, Israel}
 \email{michaelroy@tauex.tau.ac.il}

\author[Sudan Xing]{Sudan Xing}
\address[Sudan Xing]{
	Department of Mathematical and Statistical Sciences, University of Alberta, Canada}
\email{sxing@ualberta.ca}
  
  \title{On Multiple $L_p$-curvilinear-Brunn-Minkowski inequalities}
  
\begin{document}
  	\allowdisplaybreaks[4]
  	\maketitle

  	\begin{abstract}
 We construct the extension of the curvilinear summation for bounded Borel measurable sets to the $L_p$ space for multiple power parameter $\bar{\alpha}=(\alpha_1, \cdots, \alpha_{n+1})$ when $p>0$. Based on this $L_{p,\bar{\alpha}}$-curvilinear summation for sets and  concept of  {\it compression} of sets, the $L_{p,\bar{\alpha}}$-curvilinear-Brunn-Minkowski inequality for bounded Borel measurable sets and its normalized version are established. Furthermore, by utilizing the hypo-graphs for functions, we enact a brand new proof of $L_{p,\bar{\alpha}}$ Borell-Brascamp-Lieb inequality, as well as its normalized version,  for functions  containing the special case of $L_{p}$ Borell-Brascamp-Lieb inequality through the $L_{p,\bar{\alpha}}$-curvilinear-Brunn-Minkowski inequality  for sets. Moreover, we propose the multiple power $L_{p,\bar{\alpha}}$-supremal-convolution for two functions together with its properties. Last but not least, we introduce the definition of the surface area originated from the variation formula of measure in terms of the $L_{p,\bar{\alpha}}$-curvilinear summation for sets as well as  $L_{p,\bar{\alpha}}$-supremal-convolution for functions together with their corresponding  Minkowski type inequalities and isoperimetric inequalities for $p\geq1,$ etc. 
  	\end{abstract}
  	
  	\section{Introduction}
  	
  	By $\R^n$ we denote the $n$-dimensional Euclidean space with its usual inner product $\langle \cdot, \cdot \rangle$ and Euclidean norm $\|\cdot\|$. For a measurable set $A \subset \R^n$, we denote by $V_n(A)$, the volume (Lebesgue measure) of the set $A$. By $\chi_A(x)=\begin{cases}
  	1, &\text{if } x\in A,  \\
  	0, &\text{if } x\notin A,
  	\end{cases}$ we denote the characteristic function of $A$. For a subspace  \(H\in G_{n, k}\)---the
  	$k$-dimensional  Grassmannian manifold on \(\mathbb{R}^{n}\)  equipped with the Haar probability measure $\nu_{n,k}$,  its orthogonal complement $H^{\perp}\in G_{n, n-k}$ for $k\in\{0,1,\cdots,n\}.$ 
  	
  	Denote for $a, b\geq0$, $\alpha\in[-\infty, \infty]$ and $t\in[0,1]$, the $\alpha$-mean of $a, b$ as
  	\begin{equation*}
  	\begin{split}
  	M_{\alpha}^{t}(a,b) = \begin{cases}
  	\left[(1-t) a^{\alpha} + tb^{\alpha}\right]^{\frac{1}{\alpha}}, &\text{if } \alpha \neq 0, \pm \infty,\\
  	a^{1-t}b^{t}, &\text{if } \alpha = 0,\\
  	\max\{a,b\}, &\text{if } \alpha =+\infty,\\
  	\min\{a,b\}, &\text{if } \alpha = - \infty,
  	\end{cases}
  	\end{split}
  	\end{equation*}
  	if $ab>0$, 
  	and $M_{\alpha}^{t}(a,b)=0$ if $ab=0.$ 
   For Borel subsets $A,B \subset \R^n$ and $t \in (0,1)$, the Minkowski convex combination of $A$ and $B$ is the vector summation
  	\begin{eqnarray*}
  	(1-t)A + t B 
  		&=& \{(1-t)x+ty:x \in A, y \in B\}\\
  		&=&\{\big(M_1^t(x_1, y_1),\cdots,M_1^t(x_n, y_n)\big) :x=(x_1,\cdots, x_n) \in A, y=(y_1,\cdots, y_n) \in B\}. 
  	\end{eqnarray*}
  	The famous Brunn-Minkowski inequality in $\R^n$ asserts that for any measurable sets $A$ and $B$,
  	\begin{equation}\label{e:BM}
  	V_n((1-t)A+tB)\geq M_{1/n}^t (V_n(A), V_n(B)),
  	\end{equation}
  	with equality holds if and only if $A$ and $B$
  	are homothetic. Moreover, the Minkowski's first inequality associated to the variation for the volume in terms of Minkowski summation for convex bodies $A, B\in \mathcal{K}_{(o)}^n$ (convex, compact subsets in $\R^n$ with the origin $o$ in their interiors) states that
  	\begin{equation*}\label{variationconvexbody}
  	V(A,B):= \frac{1}{n} \cdot \frac{d}{d\e}V_n(A+\e B) \big\vert_{\e=0}= \frac{1}{n} \int_{\s^{n-1}} h_B(u)dS(A,u)\geq V_n(A)^{1-\frac{1}{n}}V_n(B)^{\frac{1}{n}}, 
  	\end{equation*}
  	where $S(A,\cdot)$ stands for the  surface area measure for $A$, and 
$h_{B}(u) = \max_{y \in B} \langle u,y \rangle$ for $u\in S^{n-1}$ denotes the support function of $B$ defined on $S^{n-1}$.   

For a comprehensive study of the Brunn-Minkowski inequality and its myriad applications, see the books \cite{AGM,Kold, Sh1} and the survey \cite{Gar1}, etc. 
The recent years have witnessed various directions of generalization for the Brunn-Minkowski theory, particularly,  the extension of the Brunn-Minkowski theory to the field of general measures \cite{Bobkov, BORELL, CLM, EM, GZ, KlartagMilman, KL,  Liv, LMNZ, arno1, L, P, JesusManuel, Uhrin1, Uhrin} and  functional space \cite{AFO, Ball, BCF, BrascampLieb, Co2, ColesantiFragala, Klartag, RS, liran, Rotem, Rotem2}, etc. 
  	
  	Recall that a non-negative Borel measure $\mu$ on $\R^n$ is  $\alpha$-concave for  $\alpha \in [-\infty,\infty]$ if for each $t \in (0,1)$ and all pairs of Borel sets $A,B \subset \R^n$, 
  	\begin{equation*}\label{e:alphaconcave}
  	\begin{split}
  	\mu((1-t)A+tB) &\geq 
  	\begin{cases}
  	M_{\alpha}^{t}(\mu(A), \mu(B)), &\text{if } \alpha \neq 0, \pm \infty, \\
  	\mu(A)^{1-t}\mu(B)^t, &\text{if } \alpha =0, \\
  	\max\{\mu(A),\mu(B)\}, &\text{if } \alpha =+\infty,\\
  	\min\{\mu(A),\mu(B)\}, &\text{if } \alpha =-\infty,
  	\end{cases}
  	\end{split}
  	\end{equation*}
  	whenever $\mu(A)\mu(B) > 0$, and the averages on the right-hand side are denoted as zero otherwise. The case $\alpha = 0$  is known as a $\log$-concave measure, and the case $\alpha = - \infty$ is the so-called quasi-concave measure.
  	The definition of $\alpha$-concave function  $f \colon \R^n \to \R_+$ is defined in a similar way for $\alpha\in[-\infty, \infty]$.
  	
  	The commonly used bridge connects  functional analysis with convex geometry, are the typical sets uniquely determined by the functions,  for example, the super-level of functions, Ball's body defined by functions \cite{AGM}, revolution bodies for functions \cite{AKM, Klartag,RX}, or graphs (Epigraph or Subgraph \cite{RX2}) of functions, etc. In this paper, we focus on the hypo-graph of a function $f\colon\R^n\ \to \R_+=[0,\infty)$, the area between the curve $y=f(x)$ and the $x$-axis, 
  	\[
  	\text{hyp}(f): = \{(x,r) \in \R^n \times \R_+ \colon 0 \leq r \leq f(x) \}.
  	\]
  	Obviously if $f$ is compactly supported, $\text{hyp}(f)$ is a bounded set in $\R^{n+1}.$

  	For general sets in $\R^{n+1}$, particularly,	
  	bounded Borel sets $A,B \subset \R^n \times \R_+$, Uhrin proposed the definition of the {\it curvilinear convex combination} of  $A$ and $B$ in \cite{Uhrin},   for $t \in (0,1)$, and $\alpha \in (-\infty,\infty)$, 
  	\begin{eqnarray}\label{e:uhrinoperation}
 && \!\!\!\!\!\!\!\!\!\!\!\!\!\!\!\!\!\!\!\!\!(1-t) \times_{\alpha} A +_{\alpha} t \times_{\alpha} B \\
 &:=& \left\{\left((1-t)x+ty, M_{\alpha}^t(a, b)\right) \colon (x,a) \in A, (y,b) \in B \right\}\\
  	&=&\left\{\left((1-t)(A\cap \mathbb{R}^{n})+t(B\cap \mathbb{R}^{n}), M_{\alpha}^t(a, b)\right) \colon (x,a) \in A, (y,b) \in B \right\}\label{curvisum}\\
  	&=&\left\{\left(M_1^t(x_1, y_1),\cdots,M_1^t(x_n, y_n), M_{\alpha}^t(a, b)\right) \colon (x_1,\cdots, x_n,a) \in A, (y_1,\cdots, y_n,b) \in B \right\}
  	\end{eqnarray}
  	for $ \alpha \neq 0$, and where $M_{\alpha}^t(a,b)$ for the cases $\alpha = 0,\pm \infty$ are  interpreted via continuity as $a^{1-t}b^t$, $\max\{a,b\}$, and $\min\{a,b\}$, respectively.
  	When $\alpha = 1$, it recovers the usual Minkowski convex combination in $\R^n \times \R_+$.

  	To prove the curvilinear extension for Brunn-Minkowski-Lusternik inequality for general bounded Borel measurable sets, 
  	Uhrin \cite{Uhrin2} applied the  concept of {\it compression} for a set  $A \subset \R^n \times \R_+$ which is also known as {\it shaking} originally introduced by Blaschke \cite{Blaschke, CCG} of positive measure as
  	\[
  	\tilde{A} = \text{hyp}(V_A), 
  	\]
  	where $
  	V_A(z)= V_1(A \cap (\R_+ + z)),  A \cap (z + \R_+) \neq \emptyset, z\in \R^n. 
  $
  	In another way, the compression process of a set can be viewed as taking the hypo-graph of  set segment function  defined on $\R^n$. 
  	Based on the definitions of curvilinear summation and compression for sets in $\R^{n+1}$,  Uhrin in  \cite{Uhrin} established,  a Brunn-Minkowski  type inequality for sets in $\R^{n+1}$. That is, 
  	{\it
  	for $\alpha \in (-\infty,\infty)$ and any bounded Borel subsets $A,B \subset \R^n \times \R_+$, each having finite positive volume, and $t \in (0,1)$. Suppose that  $
  		A = \tilde{A}$ and $B=\tilde{B}.
  		$
  		Then one has 
  		\begin{equation}\label{e:BMUhrin}
  		\begin{split}
  		\!\!\!\!\!\!\!V_{n+1}((1-t) \times_{\alpha} A +_{\alpha} t\times_{\alpha} B) \geq \begin{cases} M_{\frac{\alpha}{1 +n\alpha}}^t(V_{n+1}(A), V_{n+1}(B)), &\text{if } \alpha \geq - \frac{1}{n},\\
  		\min\left\{(1-t)^{\frac{1 +n\alpha}{\alpha}}V_{n+1}(A), t^{\frac{1 +n\alpha}{\alpha}}V_{n+1}(B) \right\}, &\text{if } \alpha < - \frac{1}{n}.
  		\end{cases}
  		\end{split}
  		\end{equation}
  	}
  It is obvious that  this significant inequality recovers the classic Brunn-Minkowski inequality in $\R^{n+1}$ when $\alpha=1$.
  	
  	The connection between the {\it compression} of $n+1$-dimensional set and the hypo-graph of functions  defined on $\R^n$ has the further relationships below.
  	On one hand,
  	if  $A$ and $B$ are the hypo-graphs of functions $f: \text{supp} f \rightarrow \mathbb{R}$ and $g: \text{supp} g \rightarrow \mathbb{R}$ satisfying $\text{supp} f,  \text{supp} g$ being bounded Borel sets, then 
  	by  (\ref{e:uhrinoperation}), for any $z \in \R^n$, one can obtain that 
  	\begin{eqnarray}\label{e:funvolume}
  	V_{(1-t)supp f+t supp g}(z)
  	&=&
  	V_1([(1-t) \times_{\alpha} \text{hyp}(f) +_{\alpha} t \times_{\alpha} \text{hyp}(g)] \cap (\R_+ + z))\\
  	&=& \sup_{z = (1-t) x + ty} ((1-t)f(x)^{\alpha}+ t g(y)^{\alpha})^{\frac{1}{\alpha}}.
  	\end{eqnarray}
  	On the other hand, by choosing $f = \chi_A$ and $g = \chi_B$ for any bounded Borel sets $A,B \subset \R^n$, equation \eqref{e:funvolume} takes the form $V_{(1-t)A+tB}(z)^{\alpha}=(1-t)V_A(z)^{\alpha}+tV_B(z)^{\alpha}$ since
  	\begin{align*}
  	V_1([(1-t) \times_{\alpha} \text{hyp}(\chi_A) +_{\alpha} t \times_{\alpha} \text{hyp}(\chi_B)] \cap (\R_+ + z)) &= \sup_{z = (1-t) x + ty} ((1-t)\chi_A(x)^{\alpha}+ t \chi_B(y)^{\alpha})^{\frac{1}{\alpha}}\\
  	&= \chi_{(1-t) A + t B}(z). 
  	\end{align*}
  		Noting also that $\widetilde{\text{hyp}(f)} = \text{hyp}(f)$ and $\widetilde{\text{hyp}(g)} = \text{hyp}(g)$, integrating \eqref{e:funvolume} for $z\in\R^n$, and applying (\ref{e:BMUhrin}), one obtain the Borell-Brascamp-Lieb inequality  \cite{BORELL, BrascampLieb} for $\alpha\geq-\frac{1}{n}$ and   \cite[Lemma~3.3]{DancsUhrin} for $\alpha<-1/n$. 
  	{\it
  	Let $\alpha\in(-\infty, \infty)$, $t \in (0,1)$, and $f,g,h \colon \R^n \to \R_+$ be a triple of integrable functions (with finite support) satisfying the condition 
  		\[
  		h((1-t)x+ty) \geq M_{\alpha}^t(f(x), g(y)) 
  		\]
  		for all $x,y \in \R^n$ such that $f(x)g(y)>0$.  Then 
  		\begin{equation}\label{BBLineq}
  		\int_{\R^n} h(x) dx \geq 	\begin{cases}  M_{\frac{\alpha}{1+n\alpha}}^t\left(\int_{\R^n} f(x) dx , \int_{\R^n} g(x) dx\right), &\text{if} \ \alpha \geq -\frac{1}{n},\\
  		\min\big\{(1-t)^{\frac{1+n\alpha}{\alpha}}\int_{\R^n} f(x) dx, t^{\frac{1+n\alpha}{\alpha}}\int_{\R^n} g(x) dx\big\}, & \text{if } \ \alpha<- \frac{1}{n}.
  		\end{cases}
  		\end{equation}
  	}
  	
  	Thus far we observed that the Brunn-Minkowski theory for sets in $\R^n$ is a special case for the curvilinear-Brunn-Minkowski theory in $\R^{n+1}$, which further leads to the Borell-Brascamp-Lieb inequality for functions in $\R^n$. Moreover, various other extensions of the classical Brunn-Minkowski theory have been developed in recent years, for instance, 
  	in \cite{Firey, Sh1}, Firey's generalization of the Minkowski combination---the $L_p$-Minkowski convex combination. Let $p \in [1,\infty)$ and $t \in [0,1]$. Then, for any convex bodies $K,L \subset \mathcal{K}_{(o)}^n$, and $t 
  	\in (0,1)$, the $L_p$-Minkowski combination of $K$ and $L$ with respect to coefficients $\alpha, \beta>0$, $\alpha \cdot_p K +_p \beta\cdot_p L\in \mathcal{K}_{(o)}^n$  is the convex body whose support function satisfies
  	\begin{equation}\label{e:Lpminkowskicomb}
  	h_{(1-t) \cdot_p K +_p t\cdot_p L}(u) = M_p^t(h_K(u),  h_L(u)), \quad u \in \s^{n-1}.
  	\end{equation}
  	 Equivalently, it can be reinterpreted by  {\it quasilinearzation} method \cite{Mitri} applying the {\it $L_p$ coefficients} $(C_{p,\lambda,t}, D_{p,\lambda, t})$, i.e.,
  	\begin{equation*}\label{e:ql}
  	h_{(1-t) \cdot_p K +_p t\cdot_p L}(u) = \sup_{0 \leq \lambda \leq 1}\left[C_{p,\lambda,t}h_K(u) + D_{p,\lambda,t}h_L(u) \right], \quad u \in \s^{n-1},
  	\end{equation*}
  	where 
  	\[
  	C_{p,\lambda,t}:=(1-t)^{\frac{1}{p}}(1-\lambda)^{\frac{1}{q}}\ \text{and} \ D_{p,\lambda,t}:=t^{\frac{1}{p}}\lambda^{\frac{1}{q}}.
  	\]
  	
  	What's more, in \cite{LYZ} Lutwak, Yang, and Zhang generalized the definition \eqref{e:Lpminkowskicomb} to the collection of all Borel measurable subsets of $\R^n$ that coincides with the definition proposed by Firey if the sets involved are convex bodies containing the origin in their interiors for $p\geq1$.  Given non-empty Borel measurable subsets $A$ and $B$ of $\R^n$, and $t \in (0,1)$,
  	\begin{equation}\label{e:LYZMinkowskicombo}
  	(1-t)\cdot_p A +_p t \cdot_p B = \left\{C_{p,\lambda,t}x+D_{p,\lambda,t}y \colon x \in A, y \in B, \ 0<\lambda<1\right\},
  	\end{equation}
together with the extended  $L_p$-Brunn-Minkowski inequality for non-empty measurable sets, i.e.,
  		\begin{equation*}\label{e:LYZBM}
  		V_{n}((1-t) \cdot_p A +_p t \cdot_p B) \geq M_{\frac{p}{n}}^t(V_n(A),V_n(B)).
  		\end{equation*}
Additionally, the $L_p $ Minkowski's first inequality associated to the variation for the volume in terms of $L_p$ Minkowski summation for $A,B\in\mathcal{K}_{(o)}^n$ when $p\geq1$ states that
  	\begin{equation}\label{variationconvexbodyp}
  	V_p(A,B):= \frac{p}{n} \cdot \frac{d}{d\e}V_n(A+_p\e\cdot_p B) \Big\rvert_{\e=0}= \frac{p}{n} \int_{\s^{n-1}} h_B(u)^ph_A(u)^{1-p}dS(A,u)\geq V_n(A)^{1-\frac{p}{n}}V_n(B)^{\frac{p}{n}}.
  	\end{equation}

 On the other hand, if $p\in[0,1)$, in \cite{BLYZ-1} B\"or\"oczky, Lutwak,  Yang, and Zhang posed the corresponding $L_p$ summation as the associated Aleksandrov body 
  	\begin{equation}\label{e:lpLYZdef}
  	(1-t) \cdot_p A+_p t \cdot_p B= \bigcap_{u \in \s^{n-1}} \left\{x \in \R^n \colon \langle x,u \rangle \leq M_{p}^t(h_A(u), h_B(u))  \right\}
  	\end{equation}
  	for any convex bodies $A, B \subset \mathcal{K}^n_{(o)}$, and any $t \in (0,1)$. The definitions \eqref{e:Lpminkowskicomb} and \eqref{e:lpLYZdef} agree on the range $1\leq p \leq \infty$. Moreover, the local $L_p$-Brunn-Minkowski inequality for $p\in[0,1)$ can be seen in
  	\cite{christos, Putter, global}, etc.
  	For more advances of the $L_p$-Brunn-Minkowski theory see also \cite{BHZ, BLYZ, BLYZ-1,BLYZ-2,BHT, CLM,HKL, KolM,LYZ,Lutwak1,Lutwak2,arno,liran, christos1,PT, Stan,Stan1,XiJLeng, YgZg1,Zhu}, etc.
  	
  	In view of the $L_p$ summation for sets, we notice that the $L_p$ extension of Brunn-Minkowski theory has a close relationship with the $L_p$ coefficients ($C_{p,\lambda, t}, D_{p,\lambda, t}$).
   Recall in \cite{RX, RX2}, the $L_p$ mean for numbers (or sets) with $L_p$ coefficients below.
  	Let $p\geq 1,$ $p^{-1}+q^{-1}=1$, $\alpha \in [-\infty,\infty]$, and $t,\lambda \in (0,1)$,
  	\begin{equation*}\label{t:generalizedaverages}
  	\begin{split}
  	\avg(a,b): = \begin{cases}
  	\left[C_{p,\lambda,t} a^{\alpha} + D_{p,\lambda,t}b^{\alpha}\right]^{\frac{1}{\alpha}}, &\text{if } \alpha \neq 0, \pm \infty,\\
  	a^{C_{p,\lambda,t}}b^{D_{p,\lambda,t}}, &\text{if } \alpha = 0,\\
  	\max\{a,b\}, &\text{if } \alpha =+\infty,\\
  	\min\{a,b\}, &\text{if } \alpha = - \infty,
  	\end{cases}
  	\end{split}
  	\end{equation*}
  	if $ab >0$ and $\avg(a,b) = 0$ otherwise. With the $L_p$ mean with $L_p$ coefficients in hand, the authors  in \cite{RX, RX2} further established a functional counterpart of the $L_p$ summation as well as its corresponding $L_p$ version of the Borell-Brascamp-Lieb inequality (\ref{BBLineq}).
  	The $L_p$  Borell-Brascamp-Lieb inequality for functions states that 
  	{\it if $p \geq 1$, $p^{-1} + q^{-1}=1,$ $t \in (0,1)$, $\alpha\in (-\infty, \infty)$ and   a triple of bounded integrable functions $f,g,h \colon \R^n \to \R_+$ satisfies the condition 
  		\begin{equation}\label{e:lpbblassump}
  		h\left(C_{p,\lambda,t}x+D_{p,\lambda,t}y  \right) \geq \avg(f(x),g(y))
  		\end{equation}
  		for all $x,y \in \R^n$ such that $f(x)g(y) >0$ and all $\lambda \in (0,1)$. Then
  		\begin{equation}\label{lpbblinequality}
  		\int_{\R^n}h(x) dx \geq
  		\begin{cases} 
  		M_{\frac{p\alpha}{1+n\alpha}}^{t}\left(\int_{\R^n}f(x) dx, \int_{\R^n}g(x) dx\right), &\text{if } \alpha \geq -\frac{1}{n},\\
  		\min \left\{\left[C_{p,\lambda,t}\right]^{\frac{1+n\alpha}{\alpha}}\int_{\R^n}f(x) dx, \left[D_{p,\lambda,t}\right]^{\frac{1+n\alpha}{\alpha}}\int_{\R^n}g(x) dx \right), &\text{if } \alpha <- \frac{1}{n},
  		\end{cases}
  		\end{equation}
  		for $0< \lambda<1$.
  	}
  	Choosing $p=1$ condition \eqref{e:lpbblassump}  and inequality \eqref{lpbblinequality} return to the classical Borell-Brascamp-Lieb inequality in \cite{BORELL,BrascampLieb}, each of which has found countless applications in various fields of mathematics \cite{Gar1}.  Based on the condition for this inequality, the  $L_{p,s}$-supremal convolution for two functions $f,g:\R^n\rightarrow\R_{+}$   and $s\in[-\infty,\infty]$, $f\oplus_{p,s}g\colon \R^n \to \R_+$  in \cite{RX,RX2} was defined as
  	\begin{equation}\label{e:newsupcolvolution}
  	((1-t)\otimes_{p,s}f\oplus_{p,s}t\otimes_{p,s} g)(z)=
  	\sup_{0<\lambda< 1}\left( \sup_{z=C_{p,\lambda,t}x + D_{p,\lambda,t}y} \left[C_{p,\lambda,t} f(x)^{\frac{1}{s}} +D_{p,\lambda,t}g(y)^{\frac{1}{s}}\right]^s\right)
  	\end{equation} together with its corresponding Brunn-Minkowski inequality and isoperimetric inequality for various types of measures.
  	
  	In this paper, we aim to construct the framework of $L_p$ theory for curvilinear summation for sets in $\R^{n+1}$, bases on which we further establish the corresponding $L_p$ Brunn-Minkowski type inequalities for sets. Together with this geometric inequality and the hypo-graph for functions,  we  present a brand new proof of the $L_p$ Borell-Brascamp-Lieb type inequality for functions, etc. In another way, we introduce the vector power parameter and provide the multiple version of the definition of the $L_p$ curvilinear summation for sets and its corresponding applications including $L_p$ concavity definitions for measures.  The variation formula for the measure and corresponding Minkowski type inequalities and isoperimetric inequalities with respect to multiple $L_p$ curvilinear summation for sets and $L_p$ supremal convolutions for functions are also established.
  	
  More in detail,  we introduce in Section \ref{section2} various extensions of multiple $L_p$-curvilinear combinations for sets containing  Uhrin's classical curvilinear summation for sets as special case. 
  	{\it	Let $A,B \subset (\R_+)^{n+1}$ be bounded Borel sets, $t \in (0,1)$, and power vector $\bar{\alpha} = (\alpha_1,\dots,\alpha_{n+1})$, with $\alpha_i \in [-\infty, \infty]$ for all $i=1,\dots,n+1$. 
  	For $p \geq 1$, we define the $L_{p,\bar{\alpha}}$-curvilinear summation for $A$ and $B$  as 		
	\begin{small} \begin{eqnarray*}\label{e:LpUhrinoperation0}
		&&\!\!\!\!\!\!\!\!\!\!\!\!\!\!\!\!\!\!(1-t) \otimes_{p,\bar{\alpha}} A \oplus_{p,\bar{\alpha}} t \otimes_{p,\bar{\alpha}} B\\		
		&:=&\bigcup_{0 < \lambda < 1} \left\{\left(M_{p,\alpha_1}^{(t,\lambda)}(x_1,y_1),\dots, M_{p,\alpha_{n+1}}^{(t,\lambda)}(x_{n+1},y_{n+1})\right): (x_1,\dots,x_{n+1}) \in A, (y_1,\dots,y_{n+1}) \in B\right\},    n\ge 0.
	\end{eqnarray*}
\end{small}
  \noindent While for $0<p<1$, we define the corresponding $L_{p,\bar{\alpha}}$-curvilinear summation for $A$ and $B$ with replacing $\bigcup_{0 < \lambda < 1}$ by $\bigcap_{0 < \lambda < 1}$ above. 
} We also define the {\it $L_{p,\bar{\alpha}}$-quasi-curvilinear summation} for $A$ and $B$ when $p\geq0$ in a similar way with different coefficients (See details in Section \ref{section2}) for different cases of $p$. Moreover, the monotone properties with respect to  special parameter  power  for $L_{p,\bar{\alpha}}$-curvilinear summation are also introduced.
  	
  	In Sections \ref{section3} we present a detailed proof for  the $L_{p,\bar{\alpha}}$-curvilinear-Brunn-Minkowski inequality  by induction on dimension following the ideas of Uhrin \cite{Uhrin}. The proof relied on the monotone properties under compression operator for  $L_{p,\bar{\alpha}}$-curvilinear summation for sets with special power parameter. 
  	\begin{theorem}\label{t:lpUhrinBM0}
  	Let $p \geq 1$, $t \in (0,1)$, and $\bar{\alpha} = (\alpha_1,\dots,\alpha_{n+1})$ with $\alpha_i \in (0,1]$ for each $i =1,\dots,n$.  Suppose that $A,B \subset (\R_+)^{n+1}$ are bounded Borel sets of positive measure such that $A = \tilde{A}$ and $B = \tilde{B}$. Then the following inequality holds:
  	\begin{equation*}\label{e:Uhrinlpbm}
  	\begin{split}
  	&\!\!\!\!V_{n+1}((1-t) \otimes_{p,\bar{\alpha}} A \oplus_{p,\bar{\alpha}} t \otimes_{p,\bar{\alpha}}B)\\
  	&\geq \begin{cases}
  	M_{p\gamma}^t(V_{n+1}(A),V_{n+1}(B)), &
  	\text{if } \alpha_{n+1} \geq -\left(\sum_{i=1}^n \alpha_i^{-1}\right)^{-1},\\
  	\sup \left\{ \min\left\{\left[C_{p,\lambda,t}\right]^{\frac{1}{\gamma}}V_{n+1}(A),\left[D_{p,\lambda,t}\right]^{\frac{1}{\gamma}} V_{n+1}(B)\right\} \colon 0< \lambda < 1 \right\}, &
  	\text{if } \alpha_{n+1} < -\left(\sum_{i=1}^n \alpha_i^{-1}\right)^{-1}, 
  	\end{cases}
  	\end{split}    
  	\end{equation*}
  	where 
 $
  	\gamma =\left(\sum_{i=1}^{n+1}\alpha_i^{-1}\right)^{-1}.
  $
  \end{theorem}
  \noindent When $p=1$ and  $\bar{\alpha}=\{1,\cdots,1, \alpha\}$, it recovers the curvilinear-Brunn-Minkowski inequality for sets by Uhrin \cite{Uhrin}.
  Moreover,  in Subsection \ref{normalcbm} we further present a normalized version of this $L_{p,\bar{\alpha}}$-curvilinear-Brunn-Minkowski inequality in certain circumstance. 
  	
  	In Section \ref{section4}, based on the relation of hypo-graph of functions in $\R^{n}$ with sets in $\R^{n+1}$ and the $L_p$-curvilinear-Brunn-Minkowski type inequality in Section \ref{section3}, we establish the $L_{p,\bar{\alpha}}$-Borell-Brascamp-Lieb inequality for functions, a multiple version of the $L_p$ Borell-Brascamp-Lieb inequality. While  in \cite{RX,RX2} the traditional proofs of Borell-Brascamp-Lieb type inequality encompass the revolution bodies for functions, mass transportation and classic Borell-Brascamp-Lieb inequality, we present a totally different multiple version using geometric inequalities for sets which further connects the  Brunn-Minkowski theory in convex geometric analysis and functional analysis thereof. The multiple $L_p$ Borell-Brascamp-Lieb inequality states in the following theorem.
  	  	\begin{theorem}\label{t:UhrinPLlp0} 
  		Let $p \geq 1$, $p^{-1} + q^{-1} = 1$, $t \in (0,1)$, $\bar{\alpha}=(\alpha_1,\dots, \alpha_{n+1})$, $\alpha_i \in (0,1]$ for all $i =1,\dots,n$. Suppose that $f,g,h\colon (\R_+)^n \to \R_+$ are a triple of bounded integrable functions having bounded support that satisfy the condition 
  		\begin{align}\label{multipleconvolution0}
  		h\left(M_{p,\alpha_1}^{(t,\lambda)}(x_1,y_1),\dots, M_{p,\alpha_n}^{(t,\lambda)}(x_n,y_n)\right) \geq M_{p,\alpha_{n+1}}^{(t,\lambda)}(f(x_1,\dots,x_n), g(y_1,\dots,y_n))
  		\end{align}
  		for all $x=(x_1,\cdots, x_n), y=(y_1,\cdots, y_n) \in (\R_+)^n$ such that $f(x)g(y) > 0$ and for all $\lambda \in (0,1)$. Then the following inequality holds:
  		\begin{align*}
  		&\int_{(\R_+)^n}h(x) dx\\
  		&\geq \begin{cases}
  		M_{p\gamma}^t\left(\int_{(\R_+)^n}f(x) dx,\int_{(\R_+)^n}g(x) dx \right) ,
  		&\text{if }\ \alpha_{n+1} \geq -\left(\sum_{i=1}^n \alpha_i^{-1}\right)^{-1},\\
  		\sup_{0< \lambda < 1} \min\left\{\left[C_{p,\lambda,t}\right]^{\frac{1}{\gamma}}\int_{(\R_+)^n}f(x) dx,\left[D_{p,\lambda,t}\right]^{\frac{1}{\gamma}} \int_{(\R_+)^n}g(x) dx\right\},
  		& \text{if } \ \alpha_{n+1} < -\left(\sum_{i=1}^n \alpha_i^{-1}\right)^{-1}, 
  		\end{cases}
  		\end{align*} 
  		where 
  		$
  		\gamma =\left(\sum_{i=1}^{n+1}\alpha_i^{-1}\right)^{-1}. $
  	\end{theorem}
 \noindent Similarly,  we establish a normalized version of the multiple $L_{p,\bar{\alpha}}$-Borell-Brascamp-Lieb inequality  in Subsection \ref{subsection42}. Furthermore, we present  in Subsection \ref{subsection43} a normalized  $L_{p,\bar{\alpha}}$-curvilinear-Brunn-Minkowski inequality for sets in terms of measures.
  
  In Section \ref{section5}, we focus on the variation formula of the multiple summations for sets  defined on Sections \ref{section2} firstly and establish their corresponding geometric and functional inequalities. More in detail, 
in Section \ref{subsection51},  
 the {\it $L_{p,\bar{\alpha}}$-$\mu$-surface area} for in terms of the $L_{p,\bar{\alpha}}$-curvilinear summation for sets is defined below.
 {\it
Let $\mu$ be a Borel measure on $\R^{n+1}$, $p \geq 1$, and $\bar{\alpha} \in [0, \infty]^{n+1}$. We define the $L_{p,\bar{\alpha}}$-$\mu$-surface area of a $\mu$-integrable set $A$ with respect to a $\mu$-integrable set $B$ by 
 	\[
 	S_{\mu,p, \bar{\alpha}}(A,B) := \liminf_{\e \to 0^+} \frac{	\mu(A +_{p,\bar{\alpha}} \e \times_{p,\bar{\alpha}} B) -\mu(A)}{\e},
 	\]
 }
\noindent with the following $L_p$-Minkowski's first inequality.
  {\it Let $p \in[1,\infty)$, $\bar{\alpha}=(1,\cdots,1,\alpha) \in [0, \infty]^{n+1}$, and $F \colon \R_+\to \R$ be a differentiable invertible function. Let $\mu$ be a Borel measure on $\R^n$, and assume that $\mu$ is $F(t)$-concave of any two  non-negative bounded $\mu$-measure sets $A,B\subset\R^{n+1}$ in terms of the $L_{p,\bar{\alpha}}$-curvilinear summation.
 	Then the following inequality holds: 
 	\begin{equation}\label{e:functionalMink1stconclusion0}
 	S_{\mu,p, \bar{\alpha}}(A,B) \geq S_{\mu,p, \bar{\alpha}}(A,A) + \frac{F\left(\mu(B)\right) - F\left(\mu(A)\right)}{F'\left(\mu(A) \right)}.
 	\end{equation}
 	In particular, when $\mu(A) = \mu(B)$, we obtain the following isoperimetric type inequality: 
 	\[
 	S_{\mu,p, \bar{\alpha}}(A,B) \geq S_{\mu,p, \bar{\alpha}}(A,A). 
 	\]}
 In Subsection \ref{subsection52}, based on the $L_{p,\bar{\alpha}}$ Borell-Brascamp-Lieb inequality for functions above, we introduce the {\it $L_{p, \bar{\alpha}}$ supremal-convolution} of $f, g:\R^n\rightarrow \R_+$ as 
 \begin{align*}
 &\!\!\!\!\!\!\!\!\!\!\!\!\!\!\!((1-t)\times_{p,\bar{\alpha}}f+_{p,\bar{\alpha}}t\times_{p,\bar{\alpha}}g)
 (z_1,\cdots, z_n)\\
 &=\sup_{0<\lambda<1}\left( \sup_{z_i=M_{p,\alpha_i}^{(t,\lambda)}(x_i,y_i), 1\leq i\leq n} M_{p,\alpha_{n+1}}^{(t,\lambda)}\left(f(x_1,\cdots, x_n), g(y_1,\cdots, y_n)\right)\right),
 \end{align*}
 which goes back to the corresponding $L_{p,\bar{\alpha}}$-supremal-convolution when choosing appropriate parameter $\bar{\alpha}=(1,\cdots, 1, \frac{1}{s})$ in (\ref{e:newsupcolvolution}).  
  Moreover, we give the functional counterpart in a similar method for $L_{p,\bar{\alpha}}$-$\mu$-surface area in terms of functions.  That is, {\it
 let $\mu$ be a Borel measure on $\R^n$, $p \geq 1$, and $\bar{\alpha} \in [0, \infty]^{n+1}$. We define the $L_{p,\bar{\alpha}}$-$\mu$-surface area of a $\mu$-integrable function $f \colon \R^n \to \R_+$ with respect to a $\mu$-integrable function $g$ by 
 \[
 \mathbb{S}_{\mu,p,\bar{\alpha}}(f,g) := \liminf_{\e \to 0^+} \frac{\int_{\R^n} f \oplus_{p,\bar{\alpha}} (\e \times_{p,\bar{\alpha}} g) d\mu -\int_{\R^n}f d\mu }{\e}.
 \] }
 The  $L_p$ functional counterparts of $L_p$ Minkowski's first inequality for $ \mathbb{S}_{\mu,p,\bar{\alpha}}(f,g)$ and $L_p$-isoperimetric type inequalities are also established accordingly.
 
  	\section{$L_p$-curvilinear Minkowski convex combinations}\label{section2}
  	In this Section, we introduce the $L_p$-curvilinear  summation for sets  using  the $L_p$ coefficients with two parameters to give the $L_p$-curvilinear summation with vector powers. The case for $p\geq1$ which is a natural generalization of Uhrin's curvilinear summation for sets and the  $0<p<1$ which is much more delicate for the Brunn-Minkowski type inequalities are proposed below.
  	\begin{defi}
  		Let $A,B \subset (\R_+)^{n+1}$ be bounded Borel sets, $t, \lambda \in (0,1)$, and power vector $\bar{\alpha} = (\alpha_1,\dots,\alpha_{n+1})$ with $\alpha_i \in [-\infty, \infty]$ for all $i=1,\dots,n+1$. 
  		
  		(i) For $p \geq 1$, we define the $L_{p,\bar{\alpha}}$-curvilinear summation for $A$ and $B$  as 		
  		\begin{small} \begin{eqnarray*}\label{e:LpUhrinoperation}
  			&&\!\!\!\!\!\!\!\!\!\!\!\!\!\!\!\!\!\!(1-t) \otimes_{p,\bar{\alpha}} A \oplus_{p,\bar{\alpha}} t \otimes_{p,\bar{\alpha}} B\\		
  			&:=&\bigcup_{0 < \lambda < 1} \left\{\left(M_{p,\alpha_1}^{(t,\lambda)}(x_1,y_1),\dots, M_{p,\alpha_{n+1}}^{(t,\lambda)}(x_{n+1},y_{n+1})\right): (x_1,\dots,x_{n+1}) \in A, (y_1,\dots,y_{n+1}) \in B\right\},    n\ge 0.
  			\end{eqnarray*}
  		\end{small}
For $0<p<1$, we define  the corresponding $L_{p,\bar{\alpha}}$-curvilinear summation for $A$ and $B$ replacing $\bigcup_{0 < \lambda < 1}$ by $\bigcap_{0 < \lambda < 1}$ above.
  		
  		(ii) For $p\geq1$, we define the $L_{p,\bar{\alpha}}$-quasi-curvilinear summation for $A$ and $B$  as		
  		\begin{small} \begin{eqnarray}\label{e:LpUhrinoperation}
  			&&\!\!\!\!\!\!\!\!\!\!\!\!\!(1-t) \otimes^{p,\bar{\alpha}} A \oplus^{p,\bar{\alpha}} t \otimes^{p,\bar{\alpha}} B\\
  			&:=&
  			\bigcup_{0 < \lambda < 1} \left\{\left(\min \left\{\left[C_{p,\lambda,t}\right]^{\frac{1}{\alpha_1}}x_1,\left[D_{p,\lambda,t}\right]^{\frac{1}{\alpha_1}}y_1 \right\},\cdots,\min \left\{\left[C_{p,\lambda,t}\right]^{\frac{1}{\alpha_{n+1}}}x_{n+1},\left[D_{p,\lambda,t}\right]^{\frac{1}{\alpha_{n+1}}}y_{n+1} \right\}\right)\right\},  n\ge 0,\nonumber
  			\end{eqnarray}
  		\end{small}
  where	$(x_1,\dots,x_{n+1}) \in A$ and $(y_1,\dots,y_{n+1}) \in B$.
  
  		For $0<p<1$, we define the corresponding $L_{p,\bar{\alpha}}$-quasi-curvilinear summation for $A$ and $B$  replacing $\bigcup_{0 < \lambda < 1}$ by $\bigcap_{0 < \lambda < 1}$ above.
  	\end{defi}
  In particular, without the coefficient $t$, we have \begin{eqnarray*}
  	A \oplus_{p,\bar{\alpha}}  B		
  	&=&\bigcup_{0 < \lambda < 1} \left\{\left(\left((1-\lambda)^{1/q}x_{1}^{\alpha_1}+\lambda^{1/q}y_{1}^{\alpha_1}\right)^{\frac{1}{\alpha_1}},\dots,\left((1-\lambda)^{1/q}x_{n+1}^{\alpha_{n+1}}+\lambda^{1/q}y_{n+1}^{\alpha_{n+1}}\right)^{\frac{1}{\alpha_{n+1}}}\right)\right\},    n\ge 0,
  \end{eqnarray*}
where $(x_1,\dots,x_{n+1}) \in A$ and $ (y_1,\dots,y_{n+1}) \in B,$ i.e., the $L_p$ coefficients $(C_{p,\lambda, t}, D_{p,\lambda,t})$ will be replaced accordingly as $((1-\lambda)^{1/q},\lambda^{1/q})$ in a similar way for the operations $+^{p,\bar{\alpha}}$ and $\times^{p,\bar{\alpha}}$ for all $p>0.$
  On the other hand, for the power parameter,	 if $\bar{\alpha} = (1,\dots,1,\alpha)$, we call the above definition  the {\it $L_{p, \alpha}$-curvilinear Minkowski convex combination} of $A$ and $B$. More in detail,
  	\begin{equation}\label{e:lpcuvsum}
  	\begin{split}
  	&\lpcurv\\
  	&:= \begin{cases}
  	\bigcup_{0 < \lambda < 1} \left\{\left(C_{p,\lambda,t}x+ D_{p,\lambda,t}y,\avg(a,b) \right):(x,a) \in A, (y,b) \in B \right\}, &\text{if}\ n\geq 0,\\
  	\bigcup_{0 < \lambda < 1} \left\{\avg(a,b) :a \in A, b \in B \right\}, &\text{if}\ n=0;\\
  	\end{cases}\\
  	&= \begin{cases}
  	\left(\left\{(1-t)\cdot_p(A\cap \R^n)+_p t\cdot_p (B\cap \R^n),	\bigcup_{0 < \lambda < 1}\avg(a,b) \right):(x,a) \in A, (y,b) \in B \right\}, &\text{if}\ n\geq 0,\\
  	\bigcup_{0 < \lambda < 1} \left\{\avg(a,b) :a \in A, b \in B \right\}, &\text{if}\ n=0,
  	\end{cases}
  	\end{split}
  	\end{equation}
  	and the $L_{p,\alpha}$-quasi curvilinear  Minkowski convex combination of $A$ and $B$ for $\alpha \neq 0$ has a similar definition accordingly.
  	Furthermore, if $p=1$, it  recovers the classical ($L_{1,\alpha}$) curvilinear summation and  ($L_{1,\alpha}$) quasi-curvilinear summation  in
  	(\ref{curvisum}).
  	If $ \bar{\alpha}=(1,\dots,1,1)$, it goes back to the $L_p$ Minkowski summation for $A, B\subset \R^{n+1}$, i.e., $(1-t)\cdot_p A+_p t\cdot_p B$ in (\ref{e:LYZMinkowskicombo}).

  Next we verify the monotone properties of the $L_{p,{\alpha}}$-quasi curvilinear  Minkowski summation in terms of the power $\alpha$.  	
  	\begin{prop}\label{t:properties} Let $t \in (0,1)$, and let $A,B \subset \R^n \times \R_+$ be non-empty bounded Borel measurable sets, with $V_{n+1}(A),V_{n+1}B) > 0$. Then the following statements hold.
  		\begin{enumerate}
  			
  			\item (Increasing on power) If $A = \Tilde{A}$ and $B = \Tilde{B}$, the inclusion  
  			\[
  			(1-t) \times_{p,\alpha} A +_{p,\alpha} t \times_{p,\alpha} B \supset (1-t) \times_{p,\beta} A +_{p,\beta} t \times_{p,\beta} B
  			\]
  			holds for $p \geq 1$, whenever $- \infty \leq \beta < \alpha \leq +\infty$;
  			
  			\item (Increasing on power)  If $A = \Tilde{A}$, $B = \Tilde{B}$, and in  addition, $A,B \subset (\R_+)^{n+1}$, then 
  			\[
  			(1-t) \times^{p,\alpha} A +^{p,\alpha} t \times^{p,\alpha} B \supset (1-t) \times^{p,\beta} A +^{p,\beta} t \times^{p,\beta} B
  			\]
  			holds  for $p \geq 1$
  			whenever $- \infty \leq \beta < \alpha \leq +\infty$.
  			\item (Decreasing on power) If $A = \Tilde{A}$, $B = \Tilde{B}$, and in  addition, $A,B \subset (\R_+)^{n+1}$, then 
  				\[
  			(1-t) \times^{p,\alpha} A +^{p,\alpha} t \times^{p,\alpha} B \subset (1-t) \times^{p,\beta} A +^{p,\beta} t \times^{p,\beta} B
  			\]
  				holds  for $0<p<1$
  			 whenever $- \infty \leq \beta < \alpha \leq +\infty$.

  		\end{enumerate}
  	\end{prop}
  	
  	\begin{proof}
  		(1) Choose any point $(z,r) \in (1-t) \times_{p,\beta} A +_{p,\beta} t \times_{p,\beta} B$ with $- \infty \leq \beta < \alpha \leq +\infty$.  Then by (\ref{e:lpcuvsum}), there exist $(x,a) \in A$, $(y,b) \in B$, and some $\lambda, t \in (0,1)$, such that 
  		\[
  		(z,r) = \left( C_{p,\lambda,t}x + D_{p,\lambda,t}y, M_{p,\beta}^{(t,\lambda)}(a,b) \right). 
  		\]
  It is easy to check that (for $\alpha, \beta \neq 0, \pm \infty$), 
  		\begin{align*}
  		\avg(a,b) &= \left(C_{p,\lambda,t}a^{\alpha} + D_{p,\lambda,t}b^{\alpha} \right)^{1/\alpha}\\
  		&= (M_{\lambda})^{1/\alpha}((1-\theta_{\lambda})a^{\alpha} +\theta_{\lambda}b^{\alpha})^{\frac{1}{\alpha}}\\
  		&\geq (M_{\lambda})^{1/\alpha}((1-\theta_{\lambda})a^{\beta} +\theta_{\lambda}b^{\beta})^{\frac{1}{\beta}}\\
  			&\geq (M_{\lambda})^{1/\beta}((1-\theta_{\lambda})a^{\beta} +\theta_{\lambda}b^{\beta})^{\frac{1}{\beta}}\\
  		&= M_{p,\beta}^{(t,\lambda)}(a,b),
  		\end{align*}
  		where we have used H\"older's inequality \cite{RX} $M_{\lambda}\leq 1$ when  $p\geq1$ for  \[
  		M_{\lambda}: =C_{p,\lambda,t} + D_{p,\lambda,t} \quad \text{and} \quad \theta_{\lambda} :=D_{p,\lambda,t}/M_{\lambda}.
  		\] This, together with the fact that $A = \Tilde{A}$ and $B = \Tilde{B}$, yields the desired result. The extremal cases $\alpha, \beta = 0, \pm \infty$ follow similarly.

  	 (2) As  for any $0<\lambda<1$, $0<t<1$, we have  $C_{p,\lambda,t}\leq1$ and $D_{p,\lambda, t}\leq1$ for $p\geq1$. Further since $\beta<\alpha$,  one can easily check that 
$
  	 C_{p,\lambda,t}^{\frac{1}{\alpha}}> C_{p,\lambda,t}^{\frac{1}{\beta}}
 $
  	 and 
  	  $
  	 D_{p,\lambda,t}^{\frac{1}{\alpha}}> D_{p,\lambda,t}^{\frac{1}{\beta}}$. Therefore the inclusion relation stays true for the $+^{p,\alpha}$ and  $\times^{p,\alpha}$ when  $p\geq 1$.
  
 (3) For $0<p<1$, $C_{p,\lambda,t}\geq1$ and $D_{p,\lambda, t}\geq1$, thus  the reverse inclusion relation holds.
  	    	\end{proof}
      	\begin{remark}
      		The inclusion or its inverse form is not applicable for the case of $+_{p,\alpha}$ and $\times_{p,\alpha}$ when $0<p<1$.
      		\end{remark}
      	Utilizing the method of mathematical induction on the number of power parameter, naturally we obtain the following corollary about the monotone properties of the $L_{p,\bar{\alpha}}$-curvilinear summation ($+_{p, \bar{\alpha}}, \times_{p,\bar{\alpha}}$) and $L_{p,\bar{\alpha}}$-curvilinear summation ($+^{p, \bar{\alpha}}, \times^{p,\bar{\alpha}}$) below.
  	   \begin{coro}
  	   	Let $t \in (0,1)$, and let $A,B \subset \R^n \times \R_+$ be non-empty bounded Borel measurable sets, with $V_{n+1}(A),V_{n+1}B) > 0$. Then the following statements hold.
  	   	\begin{enumerate}
  	   		
  	   		\item (Increasing on power) If $A = \Tilde{A}$, $B = \Tilde{B}$, $\bar{\alpha}=(\alpha_1,\cdots,\alpha_{n+1})$ and $\bar{\beta}=(\beta_1,\cdots,\beta_{n+1})$, 		
  	   		the inclusion  
  	   		\[
  	   		(1-t) \times_{p,\bar{\alpha}} A +_{p,\bar{\alpha}} t \times_{p,\bar{\alpha}} B \supset (1-t) \times_{p,\bar{\beta}} A +_{p,\bar{\beta} } t \times_{p,\bar{\beta}} B
  	   		\]
  	   		holds for $p \geq 1$, whenever $- \infty \leq \beta_i \leq\alpha_i \leq +\infty$ for all $1\leq i\leq n+1$;
  	   		
  	   		\item (Increasing on power)  If $A = \Tilde{A}$, $B = \Tilde{B}$, $\bar{\alpha}=(\alpha_1,\cdots,\alpha_{n+1})$, $\bar{\beta}=(\beta_1,\cdots,\beta_{n+1})$, 		 and in  addition, $A,B \subset (\R_+)^{n+1}$, then 
  	   		\[
  	   		(1-t) \times^{p,\bar{\alpha}} A +^{p,\bar{\alpha}} t \times^{p,\bar{\alpha}} B \supset (1-t) \times^{p,\bar{\beta}} A +^{p,\bar{\beta}} t \times^{p,\bar{\beta}} B
  	   		\]
  	   		holds  for $p \geq 1$ 	  whenever $- \infty \leq \beta_i \leq\alpha_i \leq +\infty$ for all $1\leq i\leq n+1$.

  	   			\item (Decreasing on power)  If $A = \Tilde{A}$, $B = \Tilde{B}$, $\bar{\alpha}=(\alpha_1,\cdots,\alpha_{n+1})$, $\bar{\beta}=(\beta_1,\cdots,\beta_{n+1})$, 		 and in  addition, $A,B \subset (\R_+)^{n+1}$, 
  	   		\[
  	   		(1-t) \times^{p,\bar{\alpha}} A +^{p,\bar{\alpha}} t \times^{p,\bar{\alpha}} B \subset (1-t) \times^{p,\bar{\beta}} A +^{p,\bar{\beta}} t \times^{p,\bar{\beta}} B
  	   		\]
  	   		holds  for $0<p<1$
  	   	 whenever $- \infty \leq \beta_i \leq\alpha_i \leq +\infty$ for all $1\leq i\leq n+1$.
  	   		\end{enumerate}	
  	   	\end{coro}		
  	\section{$L_{p}$-curvilinear-Brunn-Minkowski type inequalities}\label{section3}
 Following the concepts we proposed in   Section \ref{section2} the $L_{p,\bar{\alpha}}$-curvilinear summation ($+_{p, \bar{\alpha}}, \times_{p,\bar{\alpha}}$) and $L_{p,\bar{\alpha}}$-curvilinear summation ($+^{p, \bar{\alpha}}, \times^{p,\bar{\alpha}}$),  we aim to establish a  $L_p$-curvilinear-Brunn-Minkowski type inequality for sets for $p\geq1$. In Subsection \ref{lpcbmi}, we prove the multiple $L_{p,\bar{\alpha}}$-curvilinear-Brunn-Minkowski inequality by mathematical induction based on the properties of compression of $L_{p,\alpha}$-curvilinear summation. We also present a normalized version of $L_{p,\bar{\alpha}}$-curvilinear-Brunn-Minkowski inequality in Subsection \ref{normalcbm}.
  	\subsection{Proof of $L_{p,\bar{\alpha}}$-curvilinear-Brunn-Minkowski inequality for sets}\label{lpcbmi}
  	Firstly, we show the 1-dimensional Brunn-Minkowski type inequality in terms of  the $L_{p,\alpha}$-curvilinear summation for sets.
  	
  	\begin{lemma}\label{t:1dlemma} Let $p \geq 1$, $t \in (0,1)$, and $\alpha \in (-\infty,1]$. Let sets $K,L \subset \R_+$ be of the form 
  		\[
  		K = \bigcup_{i=1}^m [a_i,b_i] \quad L=\bigcup_{j=1}^n [c_j,d_j],
  		\]
  		where all of the intervals in the unions in $K$ and $L$ have mutually disjoint interiors and $m, n \in\mathbb{N}$ (the set of natural numbers). Then 
  		\begin{equation}\label{lpainequality}
  		V_1((1-t) \times_{p,\alpha} K +_{p,\alpha} t \times_{p,\alpha} L) \geq M_{p\alpha}^t(V_1(K), V_1(L)). 
  		\end{equation}
  	\end{lemma}

  	\begin{proof} Let us consider the case $\alpha \neq 0, - \infty$.  The proof follows by induction on the integer $N=m+n$. 
  		
  		\textbf{Step 1}: Assume that $m=n=1$ ($N=2$), in which case $K = [a,b]$ and $L = [c,d]$ for some  appropriately chosen $0\leq a<b,\ 0\leq c<d$. It is easy to check using the definition of $L_{p,\alpha}$-curvilinear summation for 1-dimension in (\ref{e:lpcuvsum}) that
  		\[
  		(1-t) \times_{p,\alpha} K +_{p,\alpha} t \times_{p,\alpha} L = \bigcup_{0 < \lambda < 1} \left[\avg(a,c), \avg(b,d) \right]. 
  		\]
  		For each fixed $\lambda \in (0,1)$, letting 
  		\[
  		M_{\lambda}: =C_{p,\lambda,t} + D_{p,\lambda,t} \quad \text{and} \quad \theta_{\lambda} :=D_{p,\lambda,t}/M_{\lambda},
  		\]
  		we obtain by Minkowski inequality that 
  		\begin{align*}
  		V_1\left(\left[\avg(a,c), \avg(b,d) \right]\right)
  		&= \avg(b,d) - \avg(a,c)\\
  		&= \left[C_{p,\lambda,t} b^{\alpha} + D_{p,\lambda,t}d^{\alpha}\right]^{\frac{1}{\alpha}} - \left[C_{p,\lambda,t} a^{\alpha} + D_{p,\lambda,t}c^{\alpha}\right]^{\frac{1}{\alpha}}\\
  		&= (M_{\lambda})^{1/\alpha}\left\{ \left[ (1-\theta_{\lambda}) b^{\alpha} + \theta_{\lambda}d^{\alpha}\right]^{\frac{1}{\alpha}} -\left[ (1-\theta_{\lambda}) a^{\alpha} + \theta_{\lambda}c^{\alpha}\right]^{\frac{1}{\alpha}}  \right\}\\
  		&\geq (M_{\lambda})^{\alpha}\left[ (1-\theta_{\lambda}) (b-a)^{\alpha} + \theta_{\lambda}(d-c)^{\alpha}\right]^{\frac{1}{\alpha}}\\
  		&= \avg(V_1(K),V_1(L)).
  		\end{align*}
  		This further implies that 
  		\begin{equation}\label{Lpbminequlaity}
  		V_1((1-t) \times_{p,\alpha} K +_{p,\alpha} t \times_{p,\alpha} L) \geq \sup_{0 < \lambda < 1} \avg(V_1(K),V_1(L))=
  		M_{p\alpha}^t(V_1(K), V_1(L)),
  		\end{equation}
  		verifying the  (\ref{lpainequality}) inequality  for case $m=n=1$. 
  		
  		\textbf{Step 2}:  Let $N > 2$ be an integer, and assume the inequality (\ref{lpainequality}) holds for the case $N -1$. Let $k,l > 0$, and set 
  		\[
  		K_1 := K \cap [0,k], \quad K_2 := K \cap [k,\infty),
  		\]
  		\[
  		L_1 := L \cap [0,l], \quad L_2: = L \cap [l,\infty). 
  		\]
  	With assumptions placed on $K$ and $L$, the sets $K_1$, $K_2$ and $L_1$, $L_2$, are mutually almost disjoint, and we see that
  		$
  		V_1(K) = V_1(K_1 \cup K_2) = V_1(K_1)+ V_1(K_2)
  		$ and
  		$
  		V_1(L) = V_1(L_1 \cup L_2) = V_1(L_1)+ V_1(L_2).
  		$

  Denote $k_1,k_2,$ and $l_1,l_2$,  the number of intervals in $K_1,K_2$, and $L_1,L_2$, respectively. Without loss of generality, we chose $k,l$ satisfying conditions below: 
  		\[
  		k_1 + l_1 \leq m+n -1, \quad k_2 + l_2 \leq m+n - 1,
  		\]
  		and 
  		\[
  		\frac{V_1(K_1)}{V_1(L_1)} = \frac{V_1(K_2)}{V_1(L_2)} = r >0. 
  		\]
  	We obtain by definition of $L_{p,\alpha}$-curvilinear summation for sets in (\ref{e:lpcuvsum}) that
  		\begin{align*}
  		(1-t) \times_{p,\alpha} K +_{p,\alpha} t \times_{p,\alpha} L &= \bigcup_{0 < \lambda < 1}\left( \bigcup_{i,j = 1}^2\left\{\avg(u,v) : (u,v) \in K_i \times L_j \right\} \right)\\
  		&\supset \bigcup_{0 < \lambda < 1 } \bigcup_{i=1}^2 \left\{\avg(u,v) : (u,v) \in K_i \times L_i \right\}\\
  		&= \bigcup_{i=1}^2 \left\{(1-t) \times_{p,\alpha} K_i +_{p,\alpha} t \times_{p,\alpha} L_i \right\},
  		\end{align*}
  		where the final union is an almost disjoint union.  Further applying the inductive step, we observe by (\ref{Lpbminequlaity}) that 
  		\begin{equation}\label{e:1lemma}
  		\begin{split}
  		&\!\!\!\!\!\!\!\!V_1((1-t) \times_{p,\alpha} K +_{p,\alpha} t \times_{p,\alpha} L) \\
  		&\geq V_1(\bigcup_{i=1}^2 (1-t) \times_{p,\alpha} K_i +_{p,\alpha} t \times_{p,\alpha} L_i)\\
  		&= V_1((1-t) \times_{p,\alpha} K_1 +_{p,\alpha} t \times_{p,\alpha} L_1)+ V_1((1-t) \times_{p,\alpha} K_2 +_{p,\alpha} t \times_{p,\alpha} L_2)\\
  		&\geq \sup_{0 < \lambda < 1} \left\{\avg(V_1(K_1),V_1(L_1)) + \avg(V_1(K_2),V_1(L_2)) \right\}.
  		\end{split}
  		\end{equation}
  		
  		For each fixed $\lambda \in (0,1)$, we notice using the assumptions placed on $k,l$, we may write 
  		\begin{equation}\label{e:2lemma}
  		\begin{split}
  		&\!\!\!\!\!\!\!\!\!\!\!\!\!\!\avg(V_1(K_1),V_1(L_1)) + \avg(V_1(K_2),V_1(L_2))\\
  		&= \left[C_{p,\lambda,t} V_1(K_1)^{\alpha} + D_{p,\lambda,t}V_1(L_1)^{\alpha}\right]^{\frac{1}{\alpha}}+ \left[C_{p,\lambda,t} V_1(K_2)^{\alpha} + D_{p,\lambda,t}V_1(L_2)^{\alpha}\right]^{\frac{1}{\alpha}}\\ &=(V_1(K_1) + V_1(K_2)) \left[C_{p,\lambda,t}  + D_{p,\lambda,t}r^{\alpha}\right]^{\frac{1}{\alpha}}\\
  		&= \left[C_{p,\lambda,t} (V_1(K_1)+V_1(K_2))^{\alpha} + D_{p,\lambda,t}(V_1(L_1)+V_1(L_2))^{\alpha}\right]^{\frac{1}{\alpha}}\\
  		&= \avg(V_1(K),V_1(L)). 
  		\end{split}
  		\end{equation}
  		It follows from \eqref{e:1lemma} and \eqref{e:2lemma}, together with  the extreme point of the function $f(\lambda) = [(1-\lambda)^{\frac{1}{q}}r + \lambda^{\frac{1}{q}}s]^{\frac{1}{\alpha}}$ with $r,s >0$, on $(0,1)$ depending on the choice of $\alpha \leq 1$ \cite{RX, RX2} that (\ref{lpainequality})  holds for $N>2.$
  	Therefore based on  Steps 1 and 2, we obtain the desired result.
  		
  	\end{proof}
  	
  		Recall the definition of compression for a bounded Borel set $\Tilde{A} \subset \R^n \times \R_+$,
  	\begin{equation}\label{e:compression}
  	\Tilde{A} = \{(x,r) \in \R^n \times \R_+ \colon 0\leq r \leq V_1(A \cap (x + \R_+)), A \cap (x + \R_+) \neq \emptyset\}. 
  	\end{equation}
  	It is easy to check that the compression operator is volume invariant,  $V_{n+1}(A) = V_{n+1}(\Tilde{A})$.  In the following, we first  further analyze another  monotone properties under compression operation $A \mapsto \Tilde{A}$ in terms of power parameters, particularly,   for the $L_{p,\alpha}$-curvilinear summation for sets, i.e., the case when $\bar{\alpha} = (1,\dots,1,\alpha)$,  $+_{p,\alpha}, +^{p,\alpha}$, etc. Then based on this property, we will present the detailed proof of  $L_{p,\alpha}$-curvilinear-Brunn-Minkowski inequality for sets in $n+1$-dimensional setting.
  \begin{theorem}\label{compressionvolume} Let $p \geq 1$, $t \in (0,1)$, and $\alpha \in [-\infty,1]$. Let $K,L \subset \R^n \times \R_+$ be bounded Borel sets such that $V_{n+1}(K), V_{n+1}(L) > 0$. Then 
  	\[
  	V_{n+1}((1-t) \times_{p,\alpha}A +_{p,\alpha} t \times_{p,\alpha}B)\geq V_{n+1}((1-t) \times_{p,\alpha}\tilde{A} +_{p,\alpha} t \times_{p,\alpha} \tilde{B}).
  	\]
  	\end{theorem}
  \begin{proof}
  	  Based on the definition for compression in (\ref{e:compression}), for any set $C \subset \R^n \times \R_+$, denote the segment projection  on $\R^n$ as
  	\[
  	C(z) = \{r \geq 0 \colon (z,r) \in C \} = C \cap (z + \R_+), \quad z \in \R^n. 
  	\]
Firstly, we verify the statement for sets  $A,B \subset \R^n \times \R_+$ of the form
  	\[
  	A = \bigcup_{m =1}^MA_m, \quad B = \bigcup_{l =1}^L B_l,
  	\]
  	where the unions are almost disjoint, and  $A_m, B_l$ are rectangular boxes with sides parallel to the coordinate axes. More precisely, we can write
  	\[
  	A_m =\left(\Pi_{i=1}^n[a_{im},c_{im}]\right) \times [a,c], \quad B_l =\left(\Pi_{i=1}^n [b_{il},d_{il}]\right) \times [b,d]
  	\]
  	for appropriate choices of constants $a_{im}<c_{im}$, $b_{il}<d_{il}$ and $a,b,c,d \geq 0$. 
  	
  	Applying Fubini's theorem,  the volume of $n+1$-dimensional set can be interpreted as the integral with respect to the segment function, i.e.,
  	\begin{eqnarray*}
  	V_{n+1}((1-t) \times_{p,\alpha}A +_{p,\alpha} t \times_{p,\alpha}B)&=& \int_{\R^n}V_{[(1-t) \times_{p,\alpha}A +_{p,\alpha} t \times_{p,\alpha}B]}(z)dz;
\\
  	V_{n+1}((1-t) \times_{p,\alpha}\tilde{A} +_{p,\alpha} t \times_{p,\alpha}\tilde{B})&=& \int_{\R^n}V_{[(1-t) \times_{p,\alpha}\tilde{A}+_{p,\alpha} t \times_{p,\alpha}\tilde{B}]}(z)dz.
  	\end{eqnarray*}
  	Based on these two identities, it suffices to show the inequality related to the segment functions
  	\[
  	V_{[(1-t) \times_{p,\alpha}A +_{p,\alpha} t \times_{p,\alpha}B]}(z) \geq V_{[(1-t) \times_{p,\alpha}\tilde{A}+_{p,\alpha} t \times_{p,\alpha}\tilde{B}]}(z)
  	\]
  	holds for almost every $z \in \R^n$. 
  	
  	Observe by definition of $L_{p,\alpha}$ curvilinear summation in (\ref{e:LpUhrinoperation}) that
  	\begin{equation}\label{e:setrepresentation1}
  	\begin{split}
  [(1-t) \times_{p,\alpha}A +_{p,\alpha} t \times_{p,\alpha}B](z)
  	&= \{r \geq 0 \colon (z,r) \in (1-t) \times_{p,\alpha}A +_{p,\alpha} t \times_{p,\alpha}B\}\\
  	&= \left\{\avg(a,b) : \left(C_{p,\lambda,t}x + D_{p,\lambda,t}y, \avg(a,b) \right)  \text{ such that } \right. \\
  	&\left. z=C_{p,\lambda,t}x + D_{p,\lambda,t}y \text{ with } (x,a) \in A, (y,b) \in B, 0 < \lambda < 1 \right\} \\
  	&=\bigcup_{0 < \lambda < 1} \left\{\bigcup_{z = C_{p,\lambda,t}x + D_{p,\lambda,t}y}\left[C_{p,\lambda,t} \times_{\alpha} A(x) +_{\alpha} D_{p,\lambda,t} \times_{\alpha} B(y) \right] \right\}.
  	\end{split}
  	\end{equation}
  	Analogously, its compression version has the following expression
  	\begin{equation}\label{e:setrepresentation2}
  	\begin{split}
  	[(1-t) \times_{p,\alpha}\tilde{A} +_{p,\alpha} t \times_{p,\alpha}\tilde{B}](z)
  	&=\bigcup_{0 < \lambda < 1} \left\{\bigcup_{z = C_{p,\lambda,t}x + D_{p,\lambda,t}y}\left[C_{p,\lambda,t} \times_{\alpha} \tilde{A}(x) +_{\alpha} D_{p,\lambda,t} \times_{\alpha} \tilde{B}(y) \right] \right\}\\
  	&=\bigcup_{0 < \lambda < 1} \left\{\bigcup_{z = C_{p,\lambda,t}x + D_{p,\lambda,t}y}\left[C_{p,\lambda,t} \times_{\alpha} \tilde{A}(\bar{x}) +_{\alpha} D_{p,\lambda,t} \times_{\alpha} \tilde{B}(\bar{y}) \right]\right\}
  	\end{split}
  	\end{equation}
  	for some $\bar{x},\bar{y} \in \R^n$. 
  	
  	The fact that 
$
  	\widetilde{A(x)} = \tilde{A}(x),$ and $\widetilde{B(y)}= \tilde{B}(y),
 $
  	together with the representations \eqref{e:setrepresentation1}, \eqref{e:setrepresentation2}, and Lemma~\ref{t:1dlemma}, yields	
  	\begin{align*}
  	V_{[(1-t) \times_{p,\alpha}A +_{p,\alpha} t \times_{p,\alpha}B]}(z)
  	&\geq V_1(\bigcup_{0 < \lambda < 1}\left[C_{p,\lambda,t} \times_{\alpha} A(\bar{x}) +_{\alpha} D_{p,\lambda,t} \times_{\alpha} B(\bar{y}) \right])\\
  	&= V_1((1-t) \times_{p,\alpha} A(\bar{x}) +_{p,\alpha} t \times_{p,\alpha} B(\bar{y}))\\
  	&\geq \sup_{0 < \lambda < 1} \avg\left(V_1(\widetilde{A(\bar{x})}), V_1(\widetilde{B(\bar{y})})\right)\\
  	&=\sup_{0 < \lambda < 1} \avg(V_1(\tilde{A}(\bar{x})),V_1(\tilde{B}(\bar{y})))\\
  	&\geq V_1\Big(\bigcup_{0 < \lambda < 1}\left[C_{p,\lambda,t} \times_{\alpha} \tilde{A}(\bar{x}) +_{\alpha} D_{p,\lambda,t} \times_{\alpha} \tilde{B}(\bar{y}) \right]\Big)\\
  	&= V_{[(1-t) \times_{p,\alpha}\tilde{A} +_{p,\alpha} t \times_{p,\alpha}\tilde{B}]}(z). 
  	\end{align*}
  	
All in all, we see that 
  	\begin{align*}
  	V_{n+1}((1-t) \times_{p,\alpha}A +_{p,\alpha} t \times_{p,\alpha}B) &= \int_{\R^n}V_{[(1-t) \times_{p,\alpha}A +_{p,\alpha} t \times_{p,\alpha}B]}(z)dz\\
  	&\geq \int_{\R^n}V_{[(1-t) \times_{p,\alpha}\tilde{A} +_{p,\alpha} t \times_{p,\alpha}\tilde{B}]}(z)dz\\
  	&= V_1((1-t) \times_{p,\alpha}\tilde{A} +_{p,\alpha} t \times_{p,\alpha} \tilde{B}),
  	\end{align*}
  	as desired.
  	
  	\end{proof}
  
   	Next we will need the following H\"{o}lder type inequalities for the generalized averages with $L_p$ coefficients to prove the $L_{p,\alpha}$-curvilinear-Brunn-Minkowski inequality.  
  Let $p \geq 1$, $p^{-1}+q^{-1}=1$, $t, \lambda \in (0,1)$, and $a,b, c,d \geq 0$. Then
  \begin{equation}\label{e:lpholder}
  \begin{split}
  \avg(a,b) \avgb(c,d)
  \geq \begin{cases}
  \avgc(ac,bd), &\text{if } \alpha + \beta \geq 0,\\
  \min\left\{\left[C_{p,\lambda,t}\right]^{\frac{1}{\gamma}}ac,\left[D_{p,\lambda,t}\right]^{\frac{1}{\gamma}} bd \right\}, &\text{if } \alpha + \beta,\alpha \cdot \beta  <0, 
  \end{cases}
  \end{split}
  \end{equation}
  where $\gamma = (\alpha \beta)/(\alpha + \beta)$. 
  Here we give a brief proof of this inequality as follows. Assume that $\alpha + \beta \geq 0$, and $\alpha, \beta \neq 0$.  The case for $\alpha+\beta\geq0$ comes from \cite{RX2}.
  We now handle the case when $\alpha + \beta, \alpha \cdot \beta < 0$. Assume that 
  \[
  \left[C_{p,\lambda,t}\right]^{\frac{1}{\gamma}}ac \leq \left[D_{p,\lambda,t}\right]^{\frac{1}{\gamma}}bd,
  \]
  and set 
  \[
  x = \frac{\left[D_{p,\lambda,t}\right]^{\frac{1}{\alpha}}b}{\left[C_{p,\lambda,t}\right]^{\frac{1}{\alpha}}a}, \quad y= \frac{\left[D_{p,\lambda,t}\right]^{\frac{1}{\beta}}d}{\left[C_{p,\lambda,t}\right]^{\frac{1}{\beta}}c}.
  \]
  Then proving \eqref{e:lpholder} amounts to establishing the inequality 
  \[
  (1+x^{\alpha})^{\frac{1}{\alpha}}(1+y^{\beta})^{\frac{1}{\beta}} \geq 1
  \]
  for $x\cdot y \geq 1$, $\alpha + \beta, \alpha \cdot \beta <0$, which origins from  \cite[Page 296]{Uhrin}, as desired.
  
  	Based on the  monotone properties of compression for $L_p$-curvilinear summations above, next we present the  high dimensional $L_{p,\alpha}$-curvilinear-Brunn-Minkowski inequality, which extends both the classical curvilinear-Brunn-Minkowski inequality by Uhrin \cite{Uhrin2}, as well as, the $L_p$-Brunn-Minkowski inequality for non-convex sets proved by Lutwak, Yang, and Zhang \cite{LYZ}.  The proof is  processed through approximation which is also a typical method used in  the proof for the Brunn-Minkowski inequality (see \cite{Stein}). 
  	
  	\begin{theorem}\label{t:lpcurvfulld} Let $p \geq 1$, $\frac{1}{p} + \frac{1}{q}=1$, $t \in (0,1)$, and $\alpha \in (-\infty,\infty)$. Suppose that $A,B \subset \R^n \times \R_+$ are bounded Borel sets of positive measure such that 
  		$
  		A = \tilde{A},  B= \tilde{B}.
  		$
  		Then the following inequality holds:
  		\begin{equation}\label{t:lpcbmfull1}
  		\begin{split}
  		&\!\!\!\!\!\!\!\!\!V_{n+1}((1-t) \times_{p,\alpha} A +_{p,\alpha} t \times_{p,\alpha}B)\\
  		&\geq \begin{cases}
  		M_{p\gamma}^t(V_{n+1}(A),V_{n+1}(B)),&\text{if } \alpha \geq -\frac{1}{n},\\
  		\sup_{0< \lambda < 1}  \min\left\{\left[C_{p,\lambda,t}\right]^{\frac{1}{\gamma}}V_{n+1}(A),\left[D_{p,\lambda,t}\right]^{\frac{1}{\gamma}} V_{n+1}(B)\right\}, &\text{if } \alpha < - \frac{1}{n},
  		\end{cases}
  		\end{split}
  		\end{equation}
  		where $\gamma=\frac{\alpha}{1+n\alpha}$.
  	
  	\end{theorem}
  	
  	\begin{proof}
  		
  		To begin with, we consider  the box-type sets $A$ and $B$ with the form
  		\[
  		A = \bigcup_{m = 1}^M (A_m \times [0,a_m]), \quad B= \bigcup_{l=1}^L(B_l \times [0,b_l]),
  		\]
  		where the $A_m$, $B_l \subset \R^n$ are rectangular boxes with sides parallel to the coordinates axes, and the union are  almost disjoint when $\alpha\geq-\frac{1}{n}$.  The proof is similar to the 1-dimensional case  by induction on $N=M+L$ in two steps below.  
  		
  	\textbf{Step 1:} Assume that $N=2 \ (M=L=1)$.  Then $A$ and $B$ are of the form 
  		\[
  		A = \left(\prod_{i=1}^n[a_i,c_i]\right) \times [0,a], \quad B= \left(\prod_{i=1}^n[b_i,d_i]\right) \times [0,b].
  		\]
  		Notice that 
  		\begin{align*}
  		&\!\!\!\!\!\!\!\!\!(1-t) \times_{p,\alpha} A +_{p,\alpha} t \times_{p,\alpha}B\\
  		&=\bigcup_{0 < \lambda < 1} \left(\prod_{i=1}^n \big[C_{p,\lambda,t}a_i + D_{p,\lambda,t}b_i,C_{p,\lambda,t}c_i + D_{p,\lambda,t}d_i \big] \right) \times \left[0, \avg(a,b) \right]\\
  		&= \bigcup_{0 < \lambda <1}\left(C_{p,\lambda,t} \prod_{i=1}^n[a_i,c_i] + D_{p,\lambda,t} \prod_{i=1}^n[b_i,d_i] \right) \times \left[0, \avg(a,b) \right].
  		\end{align*}
  		
  		For each $\lambda \in (0,1)$, by applying Fubini's theorem and the  Brunn-Minkowski inequality in $\R^n$ (\ref{e:BM}), we obtain
  		\begin{align*}
  		&\!\!\!\!\!\!\!\!\!\!V_{n+1}\left(\left(C_{p,\lambda,t} \prod_{i=1}^n[a_i,c_i] + D_{p,\lambda,t} \prod_{i=1}^n[b_i,d_i] \right) \times \left[0, \avg(a,b) \right]\right)\\
  		&= V_{n}\left(\left(C_{p,\lambda,t} \prod_{i=1}^n[a_i,c_i] + D_{p,\lambda,t} \prod_{i=1}^n[b_i,d_i] \right)\right) \cdot V_1(\left[0, \avg(a,b) \right])\\
  		&\geq \avgb\left(V_{n}(\prod_{i=1}^n[a_i,c_i]), V_{n}(\prod_{i=1}^n[b_i,d_i])\right) \cdot \avg(a,b),
  		\end{align*}
  		with $\beta = \frac{1}{n}$. 
  		
  		Therefore, we have
  		\begin{align*}
  		V_{n+1}\left((1-t) \times_{p,\alpha} A +_{p,\alpha} t \times_{p,\alpha}B\right) \geq \sup_{0 < \lambda <1} \left\{\avgb\left(V_n(\prod_{i=1}^n[a_i,c_i]), V_n\big(\prod_{i=1}^n[b_i,d_i]\big)\right) \cdot \avg(a,b) \right\},
  		\end{align*}
  		with $\beta = \frac{1}{n}$. Applying the H\"older type inequality \eqref{e:lpholder} with the above choice of $\beta = \frac{1}{n}$ yields the result in this case. 
  		
  		\textbf{Step 2:} Assume that we have established the result in the case $M+L< N$ for some $N >2$.  Let $\bar{H}$ be a coordinate hyperplane in $\R^n$ and set $H = \bar{H} \times \R_+$. We cut $A$ and $B$ by translates $H+x$ and $H+y$ for some $x,y \in \R^n$. Denote by $(H + x)^{-}$, $(H+x)^+$ the closed half-spaces of $\R^n \times \R_+$ determined by $x$, and similarly define $(H+y)^-$ and $(H+y)^+$.  
  	Consider the sets 
 \begin{eqnarray*}
  		A_1 &= A \cap (H+x)^-, \quad A_2 = A \cap (H+x)^+,\\
  		B_1 &= B \cap (H+y)^-, \quad B_2 = B \cap (H+y)^+.
  		\end{eqnarray*}
  		Notice that 
  		\[
  		A = A_1 \cup A_2, \quad B = B_1 \cup B_2,
  		\]
  		where the unions are almost disjoint. 
  		Denote by $m_1,m_2$ and $l_1,l_2$ the numbers of rectangular boxes comprising $A_1,A_2$ and $B_1,B_2$, respectively.  We can select $\bar{H}$, $x$ and $y$ in such a way that 
  		\[
  		m_1 + l_1 \leq M+L -1, \quad m_2 + l_2 \leq M + L -1,
  		\]
  		and 
  		\[
  		\frac{V_{n+1}(A_1)}{V_{n+1}(B_1)} = 	\frac{V_{n+1}(A_2)}{V_{n+1}(B_2)} = r >0. 
  		\]
  		Notice 
  		\begin{align*}
  		&\!\!\!\!\!\!\!\!\!\!\!\!\!(1-t) \times_{p,\alpha}A +_{p,\alpha} t \times_{p,\alpha}B\\
  		&= \bigcup_{0 < \lambda <1} \left(\bigcup_{i,j=1}^2 \left\{\left(C_{p,\lambda,t}x + D_{p,\lambda,t}y, \avg(a,b) \right) \colon ((x,a),(y,b)) \in A_i \times B_j \right\}  \right)\\
  		&\supset \bigcup_{0 < \lambda <1} \left(\bigcup_{i=1}^2 \left\{\left(C_{p,\lambda,t}x + D_{p,\lambda,t}y, \avg(a,b) \right) \colon ((x,a),(y,b)) \in A_i \times B_i \right\} \right)\\
  		&= \bigcup_{i=1}^2 (1-t) \times_{p,\alpha}A_i +_{p,\alpha} t \times_{p,\alpha} B_i
  		\end{align*}
  		where the last union is almost disjoint. 
  		
  		By taking volumes, and using the inductive hypothesis, we see that 
  		\begin{align*}
  		&\!\!\!\!\!\!V_{n+1}((1-t) \times_{p,\alpha}A +_{p,\alpha} t \times_{p,\alpha}B) \\
  		&\geq V_{n+1}\left(\bigcup_{i=1}^2 (1-t) \times_{p,\alpha}A_i +_{p,\alpha} t \times_{p,\alpha} B_i\right)\\
  		&=\sum_{i=1}^2V_{n+1}((1-t) \times_{p,\alpha}A_i +_{p,\alpha} t \times_{p,\alpha} B_i)\\
  		&\geq  \begin{cases}
  		\sup_{0< \lambda < 1} \sum_{i=1}^2\avgc(V_{n+1}(A_i),V_{n+1}(B_i)), &\text{if } \alpha \geq -\frac{1}{n},\\
  		\sup_{0< \lambda < 1} \sum_{i=1}^2 \min\left\{\left[C_{p,\lambda,t}\right]^{\frac{1}{\gamma}}V_{n+1}(A_i),\left[D_{p,\lambda,t}\right]^{\frac{1}{\gamma}} V_{n+1}(B_i)\right\}, &\text{if } \alpha < - \frac{1}{n},
  		\end{cases}
  		\end{align*}
  		where $\gamma=\frac{\alpha}{1+n\alpha}$.
  		
  		For each $\lambda \in (0,1)$, the assumptions on $m_i,l_i$, $i=1,2$ yield, for $\alpha \geq -\frac{1}{n}$,
  		\begin{align*}
  		\sum_{i=1}^2\avgc(V_{n+1}(A_i),V_{n+1}(B_i)) &= (V_{n+1}(A_1)+V_{n+1}(A_2))\avgc(1,r)\\
  		&=\avgc\big(V_{n+1}(A_1)+V_{n+1}(A_2),V_{n+1}(B_1)+V_{n+1}(B_2)\big) \\
  		&= \avgc\big(V_{n+1}(A),V_{n+1}(B)\big).
  		\end{align*}
  		With the maximal point for $0<\lambda<1$, we obtain the inequality for $\alpha\geq -\frac{1}{n}$ in (\ref{t:lpcbmfull1}).
  		
  		A similar argument works for the case $\alpha < -\frac{1}{n}$. This completes the proof for sets which are comprised of unions of almost disjoint rectangular boxes in $\R^n$ and an interval in $\R_+$. 
  			To complete the proof, we use standard density considerations (see, for example, \cite{Stein, Uhrin}). 
  	\end{proof}
  	
  	Combining Theorems \ref{compressionvolume} and  \ref{t:lpcurvfulld}, we obtain the following corollary, which shows that,  Theorem~\ref{t:lpcurvfulld} is indeed a sharper form of the $L_p$-Brunn-Minkowski inequality for  sets before compression. 
  	
  	\begin{coro}
  		Let $p \geq 1$, $p^{-1}+q^{-1} = 1$, $t \in (0,1)$, and  $-\frac{1}{n} \leq \alpha \leq 1$. Then, for any pair of bounded Borel sets $A,B \subset \R^n \times \R_+$, each of positive volume, one has 
  		\begin{equation*}
  		\begin{split}
  		V_{n+1}((1-t) \times_{p,\alpha} A +_{p,\alpha} t \times_{p,\alpha} B)&\geq V_{n+1}((1-t) \times_{p,\alpha} \tilde{A} +_{p,\alpha} t \times_{p,\alpha} \Tilde{B})\\
  		&\geq M_{p\gamma}^t(V_{n+1}(A),V_{n+1}(B)),
  		\end{split}
  		\end{equation*}
  		where $\gamma = \frac{\alpha}{1+n\alpha}$.
  	\end{coro}

  	Naturally, by repeating the above Theorem with respect to the multiple power parameters for $L_{p,\alpha}$-curvilinear-Brunn-Minkowski inequality,  we obtain the following theorem in terms of the $L_{p,\bar{\alpha}}$-curvilinear-Brunn-Minkowski inequality for sets with the vector power $\bar{\alpha} = (\alpha_1,\dots,\alpha_{n+1})$.

  	\begin{theorem}\label{t:lpUhrinBM}
  		Let $p \geq 1$, $t \in (0,1)$, and $\bar{\alpha} = (\alpha_1,\dots,\alpha_{n+1})$ with $\alpha_i \in (0,1]$ for each $i =1,\dots,n$.  Suppose that $A,B \subset (\R_+)^{n+1}$ are bounded Borel sets of positive measure such that $A = \tilde{A}$ and $B = \tilde{B}$. Then the following inequality holds:
  		\begin{equation*}\label{e:Uhrinlpbm}
  		\begin{split}
  		&\!\!\!\!V_{n+1}((1-t) \otimes_{p,\bar{\alpha}} A \oplus_{p,\bar{\alpha}} t \otimes_{p,\bar{\alpha}}B)\\
  		&\geq \begin{cases}
  		M_{p\gamma}^t(V_{n+1}(A),V_{n+1}(B)), &
  		\text{if } \alpha_{n+1} \geq -\left(\sum_{i=1}^n \alpha_i^{-1}\right)^{-1},\\
  		\sup_{0< \lambda < 1}  \min\left\{\left[C_{p,\lambda,t}\right]^{\frac{1}{\gamma}}V_{n+1}(A),\left[D_{p,\lambda,t}\right]^{\frac{1}{\gamma}} V_{n+1}(B)\right\}, &
  		\text{if } \alpha_{n+1} < -\left(\sum_{i=1}^n \alpha_i^{-1}\right)^{-1}, 
  		\end{cases}
  		\end{split}    
  		\end{equation*}
  		where 
  	$
  		\gamma =\left(\sum_{i=1}^{n+1}\alpha_i^{-1}\right)^{-1}.
  $
  	\end{theorem}
  	
  	\begin{remark}
  		  If $\bar{\alpha}=(1,\cdots,1)$, the $L_{p,\bar{\alpha}}$-curvilinear summation recovers the $L_p$ Minkowski summation by Lutwak, Yang and Zhang \cite{LYZ} and the above $L_{p,\bar{\alpha}}$-curvilinear-Brunn-Minkowski inequality goes back to the $L_p$ Brunn-Minkowski inequality for sets.
  		\end{remark}

  	\subsection{Normalized  $L_p$-curvilinear-Brunn-Minkowski inequalities}\label{normalcbm}
  	In the following we will examine the normalized version of the $L_{p,\bar{\alpha}}$-curvilinear-Brunn-Minkowski inequality in previous Subsection.

  Although all results in this section may be formulated for general bounded Borel sets in $\R^n \times \R_+$, due to standard density considerations, it suffices to establish all results for compact $F_{\sigma}$-sets. To this end, we assume that $A,B \subset \R^n \times \R_+$  of the form
  \begin{equation*}\label{e:fsigma}
  A= \bigcup_{m=1}^M (\bar{A}_m \times A_m), \quad  B = \bigcup_{l=1}^M (\bar{B}_l \times B_l), \quad M \leq \infty,
  \end{equation*}
  where $\bar{A}_m, \bar{B}_m \subset \R^n$ and $A_m, B_m \subset \R_+$ are compact sets for all $m,l$. Moreover, we have the following  identities by definitions of $L_{p,\alpha}$ operations
  \begin{equation*}\label{e:fsigmaidt}
  \begin{split}
  &\lpcurv
  = \bigcup_{m,l=1}^M \bigcup_{\lambda \in (0,1) \cap \Q} \big(C_{p,\lambda,t} \times_{\alpha}(\bar{A}_m \times A_m) +_{\alpha} D_{p,\lambda,t}\times_{\alpha} (\bar{B}_l \times B_l)\big);\\
  &\lpmin=  \bigcap_{m,l=1}^M \bigcup_{\lambda \in (0,1) \cap \Q} \big(C_{p,\lambda,t} \times^{\alpha}(\bar{A}_m \times A_m) +^{\alpha} D_{p,\lambda,t}\times^{\alpha} (\bar{B}_l \times B_l)\big),
  \end{split}
  \end{equation*}
  where $\mathbb{Q}$ denotes the set of all rational numbers.
   It is not difficult to see that each $C_{p,\lambda,t} \times_{\alpha}(\bar{A}_m \times A_m) +_{\alpha} D_{p,\lambda,t}\times_{\alpha} (\bar{B}_l \times B_l)$ and each $C_{p,\lambda,t} \times^{\alpha}(\bar{A}_m \times A_m) +^{\alpha} D_{p,\lambda,t}\times^{\alpha} (\bar{B}_l \times B_l)$ are compact sets, and consequently, $\lpcurv$ and $\lpmin$ are bounded $F_{\sigma}$-subsets of $\R^n \times \R_+$.

  Next we consider the generalized section function and the corresponding compression.
  Let $\bar{H} = H \times \R_+$ for  $H \subset G_{n,k}$ when $k\in\{0,1,\cdots, n\}$. Fix any bounded Borel measurable subset $A \subset \R^n \times \R_+$, and the generalized segment function with respect to $H$ on $H^{\perp},$ $V_{A,H} \colon \h \to \R_+$ is defined by 
  \begin{equation*} \label{e:subspacefunctional}
  V_{A,H}(y) = V_{k+1}(A \cap (\bar{H}+y)).
  \end{equation*}
  When $k =0$, $H=\{o\}$, and $\h = \R^n$,  $V_{A,H}(\cdot)$ recovers $V_A(\cdot)$, the classical segment function. 
  It is easy to see that the function $V_{A,H}$ is integrable and bounded. 
  
  Denote the {\it optimal norm $\|V_{A,H}\|_{\infty}$, normalized super-level set $C_r(V_{A,H})$ for $V_{A,H}$, and the normalized compression $A_H$}, correspondingly as  \begin{equation}\label{e:criticalquantities}
  \begin{split}
  &\|V_{A,H}\|_{\infty} = \sup_{y \in \h} V_{A,H}(y),\\
  &C_r(V_{A,H}) = \left\{y \in \h \colon V_{A,H}(y) \geq r \|V_{A,H}\|_{\infty} \right\}, \quad 0 \leq r \leq 1,\\
  &A_H=\left\{(y,r) \in \h \times \R_+ \colon 0\leq r \leq \frac{V_{A,H}(y)}{\|V_{A,H}\|_{\infty}}, \ A \cap(\bar{H} + y) \neq \emptyset \right\}. 
  \end{split}
  \end{equation}
  When $k=0$,   we denote the above as $\|V_{A_0}\|$, $C_r(V_{A_0})$ and $A_0$, respectively.

  	\begin{prop}\label{t:refinementfulld}
  		Let $p \geq 1$, $p^{-1}+q^{-1}=1$, $t \in (0,1)$, and $\alpha \in (-\infty,\infty)$. Then, for any bounded $F_{\sigma}$-sets $A,B \subset \R^n \times \R_+$ of positive volume,
  		\begin{equation*}\label{e:refinement1}
  		\begin{split}
  		V_{n+1}((1-t) \times_{p,\alpha} A_0 +_{p,\alpha} t \times_{p,\alpha} B_0)&\geq V_{n+1}((1-t) \times_{p,-\infty} A_0 +_{p,-\infty} t \times_{p, -\infty} B_0)\\
  		&\geq \int_0^1 V_{n}((1-t) 
  		\cdot_p C_r\left(V_{A_0}\right)  +_p t \cdot_p C_r\left(V_{B_0}\right))dr.
  		\end{split}
  		\end{equation*}
  	\end{prop}
  	
  	\begin{proof}
  		All of the sets $A_0,B_0, C_r(V_{A_0})$, and $C_r(V_{B_0})$ are $F_{\sigma}$-sets, and $A_0, B_0$ satisfy the compression condition 
  		\[
  		A_0 = \tilde{A}_0, \quad B_0 = \tilde{B}_0. 
  		\]
  Let $z \in \R^n$ be such that $[(1-t) \times_{p,-\infty} A_0 +_{p,-\infty} t \times_{p,-\infty} B_0]
  		\cap (\R_+ \cap z) \neq \emptyset$. Then we see that		
  		\begin{align*}
  		&\!\!\!\!\!\!\!\!\![(1-t) \times_{p,-\infty} A_0 +_{p,-\infty} t \times_{p,-\infty} B_0] \cap(\R_+ + z) \\
  		&=[(1-t) \times_{p,-\infty} \tilde{A}_0 +_{p,-\infty} t \times_{p,-\infty} \tilde{B}_0] \cap (\R_+ + z)\\
  		&= \left(\bigcup_{\lambda \in \Q \cap (0,1)} C_{p,\lambda,t} \times_{-\infty} \tilde{A}_0 +_{-\infty} D_{p,\lambda,t}\times_{-\infty} \tilde{B}_0\right)\cap (\R_+ + z)\\
  		&=\bigcup_{\lambda \in \Q \cap (0,1)}\left\{\left[C_{p,\lambda,t} \times_{-\infty} \tilde{A}_0 +_{-\infty} D_{p,\lambda,t} \times_{-\infty} \tilde{B}_0\right] \cap (\R_+ + z) \right\}\\
  		&=\bigcup_{\lambda \in \Q \cap (0,1)} \left[\left\{\left(C_{p,\lambda,t}x + D_{p,\lambda,t}y,\min\{a,b\} \right) \colon (x,a) \in \tilde{A}_0, (y,b) \in \tilde{B}_0\right\}\cap (\R_+ + z)\right] \\
  		&=\bigcup_{\lambda \in \Q \cap (0,1)} \bigcup_{z = C_{p,\lambda,t}x + D_{p,\lambda,t}y}  \left\{\left(z,0 \right) \times \left[0,\min\left\{\frac{V_{A_0}(x)}{\|V_{A_0}\|_{\infty}}, \frac{V_{B_0}(y)}{\|V_{B_0}\|_{\infty}} \right\} \right]   \right\}.
  		\end{align*}

  		By taking volumes, we see that, for each such $z \in \R^n$,
  		\begin{align*}
  		&\!\!\!\!\!\!\!\!\!V_1([(1-t) \times_{p,-\infty} A_0 +_{p,-\infty} t \times_{p,-\infty} B_0] \cap(\R_+ + z)) \\
  		&=\sup_{0 < 
  			\lambda <1} \left[ \sup_{z = C_{p,\lambda,t}x + D_{p,\lambda,t}y} \min\left\{\frac{V_{A_0}(x)}{\|V_{A_0}\|_{\infty}},\frac{V_{B_0}(y)}{\|V_{B_0}\|_{\infty}}  \right\}\right],
  		\end{align*}
  		which implies the inclusion relationship 
  		\begin{equation}\label{e:inclusion1}
  		\begin{split}
  		C_r \left(V_{(1-t) \times_{p,-\infty} A_0 +_{p,-\infty} t \times_{p,-\infty} B_0} \right) \supset (1-t) 
  		\cdot_p C_r\left(V_{A_0}\right)  +_p t \cdot_p C_r\left(V_{B_0}\right) .
  		\end{split}
  		\end{equation}
  	Fubini's theorem together with the inclusion \eqref{e:inclusion1} yields
  		\begin{align*}
  		&\!\!\!\!\!\!\!\!\!V_{n+1}((1-t) \times_{p,-\infty} A_0 +_{p,-\infty} t \times_{p,-\infty} B_0)\\
  		&=\int_{\R^n} V_1(\left[(1-t) \times_{p,-\infty} A_0 +_{p,-\infty} t \times_{p,-\infty} B_0\right] \cap (\R_+ + z))dz\\
  		&= \int_0^1 V_n\left(C_r \left(V_{(1-t) \times_{p,-\infty} A_0 +_{p,-\infty} t \times_{p,-\infty} B_0} \right)\right)dr\\
  		&\geq \int_0^1 V_n\left((1-t) 
  		\cdot_p C_r\left(V_{A_0}\right)  +_p t \cdot_p C_r\left(V_{B_0}\right)  \right)dr.
  		\end{align*}
  		Finally, utilizing the monotone properties in  Proposition~\ref{t:properties}  (1), we have 
  		\begin{align*}
  		V_{n+1}((1-t) \times_{p,\alpha} A_0 +_{p,\alpha} t \times_{p,\alpha} B_0)&\geq  V_{n+1}((1-t) \times_{p,-\infty} A_0 +_{p,-\infty} t \times_{p, -\infty} B_0)\\
  		&\geq \int_0^1 V_n((1-t) 
  		\cdot_p C_r\left(V_{A_0}\right)  +_p t \cdot_p C_r\left(V_{B_0}\right)) dr,
  		\end{align*}		
  		completing the proof.
  	\end{proof} 	
 Next we will obtain the following {\it normalized} $L_{p,\bar{\alpha}}$-curvilinear-Brunn-Minkowski inequality for sets when $\bar{\alpha}=(1,\cdots,1,\alpha)$. 
  	
  	\begin{theorem}\label{t:normalizedlp}
  		Let $p \geq 1$, $p^{-1} + q^{-1} =1$,  and $t \in (0,1)$. Suppose that $A,B \subset \R^n \times \R_+$ are bounded $F_{\sigma}$-sets having positive volume such that $A = \tilde{A}$ and $B = \tilde{B}$. Then, for any $\alpha \in (-\infty,\infty)$, 
  		\begin{align*}
  		&\!\!\!\!\!\!\!\!\!V_{n+1}(\lpcurv) \cdot M_{-p\alpha}^{t}\left(\|V_A\|_{\infty}^{-1},\|V_B\|_{\infty}^{-1} \right)\\
  		&\geq \int_0^1 V_{n}((1-t) 
  		\cdot_p C_r\left(V_{A}\right)  +_p t \cdot_p C_r\left(V_{B}\right)) dr.
  		\end{align*}	
  	\end{theorem}
  	
  	\begin{proof} Arguing as in proof of Proposition~\ref{t:refinementfulld} and using the fact that $A = \tilde{A}, B = \tilde{B}$, we see that, for each $z \in \R^n$ such that 
  $
  		[\lpcurv] \cap(\R_+ + z) \neq \emptyset,
  	$
  and
  		\begin{align*}
  		&\!\!\!\!\!\!\!\!\!\!\!\!\!\![\lpcurv]\cap(\R_+ + z)
  		\\
  		&=\bigcup_{\lambda \in \Q \cap (0,1)} \bigcup_{z = C_{p,\lambda,t}x + D_{p,\lambda,t}y} \left\{\left(z,0 \right) \times\left[0, \avg(V_A(x),V_B(y)) \right] \right\}.
  		\end{align*}
  	
  	Further taking volumes, we conclude that for each fixed $z \in \R^n$,
  		\begin{align*}
  		&V_1([\lpcurv] \cap(\R_+ + z))=\sup_{0 < 
  			\lambda <1}  \sup_{z = C_{p,\lambda,t}x + D_{p,\lambda,t}y} \avg\left(V_A(x),V_B(y)\right).
  		\end{align*}
  		Applying the inequality (\ref{e:lpholder}) with $\beta = -\alpha$, we conclude the relation between the segment function for $L_{p,\alpha}$-curvilinear summation for sets and the normalized segment function for $L_{p,\alpha}$-curvilinear summation below, i.e.,
  		 \begin{align*}
  		&\!\!\!\!\!\!\!\!V_1([\lpcurv] \cap(\R_+ + z)) \cdot \sup_{0 < \lambda <1}\left(M_{p,-\alpha}^{(t,\lambda)}\left(\|V_A\|_{\infty}^{-1},\|V_B\|_{\infty}^{-1} \right)\right)\\
  		&\geq \sup_{0 < 
  			\lambda <1} \left[ \sup_{z = C_{p,\lambda,t}x + D_{p,\lambda,t}y} \avg\left(V_A(x),V_B(y) \right) \cdot M_{p,-\alpha}^{(t,\lambda)}\left(\|V_A\|_{\infty}^{-1},\|V_B\|_{\infty}^{-1} \right) \right]\\
  		&\geq \sup_{0 < 
  			\lambda <1} \left[ \sup_{z = C_{p,\lambda,t}x + D_{p,\lambda,t}y} \min\left\{\frac{V_A(x)}{\|V_A\|_{\infty}},\frac{V_B(y)}{\|V_B\|_{\infty}}  \right\}\right]\\
  		&=V_1([(1-t) \times_{p,-\infty} A_0 +_{p,-\infty} t \times_{p,-\infty} B_0] \cap(\R_+ + z)). 
  		\end{align*}
  		
  		Integrating this inequality for all $z\in\R^n$ and using Fubini's theorem, we obtain 
  		\begin{align*}
  		&\!\!\!\!\!\!\!\!V_{n+1}(\lpcurv) \cdot M_{-p\alpha}^{t}\left(\|V_A\|_{\infty}^{-1},\|V_B\|_{\infty}^{-1} \right)\\
  		&= V_{n+1}(\lpcurv)\cdot\sup_{0 < \lambda <1}\left(M_{p,-\alpha}^{(t,\lambda)}\left(\|V_A\|_{\infty}^{-1},\|V_B\|_{\infty}^{-1} \right)\right)\\
  		&= \int_{\R^n} V_1([\lpcurv] \cap(\R_+ + z)) dz\cdot \sup_{0 < \lambda <1}\left(M_{p,-\alpha}^{(t,\lambda)}\left(\|V_A\|_{\infty}^{-1},\|V_B\|_{\infty}^{-1} \right)\right)\\
  		&\geq \int_{\R^n}\sup_{0 < 
  			\lambda <1} \left[ \sup_{z = C_{p,\lambda,t}x + D_{p,\lambda,t}y} \min\left\{\frac{V_A(x)}{\|V_A\|_{\infty}},\frac{V_B(y)}{\|V_B\|_{\infty}}  \right\}\right] dz\\
  		&=\int_{\R^n} V_1(\left[(1-t) \times_{p,-\infty} A_0 +_{p,-\infty} t \times_{p,-\infty} B_0\right] \cap (\R_+ + z))dz\\
  		&= \int_0^1 V_n(C_r \left(V_{(1-t) \times_{p,-\infty} A_0 +_{p,-\infty} t \times_{p,-\infty} B_0} \right))dr\\
  		&\geq \int_0^1 V_n((1-t) 
  		\cdot_p C_r\left(V_{A}\right)  +_p t \cdot_p C_r\left(V_{B}\right)) dr,
  		\end{align*}
  		by Proposition \ref{t:refinementfulld}, 
  		 completing the proof. 
  		
  	\end{proof}
  	
  Moreover, we extend the normalized above $L_{p,\alpha}$-curvilinear-Brunn-Minkowski inequality to more general space with power parameters  $\bar{\alpha}=(1,\cdots,1,\alpha)$. 
  	
  	\begin{theorem}\label{t:sectionallpinequality} Let $p \geq 1$, $p^{-1}+q^{-1}=1$, $t \in (0,1)$, and let $A,B \subset \R^n \times \R_+$ be bounded $F_{\sigma}$-sets of positive volume such that $A = \tilde{A}$ and $B = \tilde{B}$. Let $H \subset G_{n,k}$ for $k\in\{0,\cdots,n\}$, and let $\|V_{A,H}\|_{\infty}, \|V_{B,H}\|_{\infty}$ and $C_r(V_{A,H})$, $C_r(V_{B,H})$ be  defined as in \eqref{e:criticalquantities}. Then, for any $\alpha, \beta \in (-\infty,\infty)$ with $\alpha + \beta \geq 0$, we obtain
  			\begin{equation*}\label{e:strongest}
  		\begin{split}
  		&\!\!\!\!\!\!\!\!\!\!V_{n+1}(\lpcurv) \cdot  M_{p\beta}^t(\|V_{A,H}\|_{\infty}^{-1}, \|V_{B,H}\|_{\infty}^{-1})\\
  		&\geq \begin{cases}
  		V_{n-k+1}((1-t) \times_{p,\delta} A_H+_{p,\delta} t \times_{p,\delta} B_H), &\text{if } \frac{\alpha\beta}{\alpha + \beta} \geq -\frac{1}{k},\\
  		V_{n-k+1}((1-t) \times^{p,\delta} A_H+^{p,\delta} t \times^{p,\delta} B_H), &\text{if } \frac{\alpha\beta}{\alpha + \beta} < - \frac{1}{k},
  		\end{cases}
  		\end{split}    
  		\end{equation*}
	where $\delta = (\alpha^{-1} + \beta^{-1} + k)^{-1}$.
  		
  	\end{theorem}
  	
  	\begin{proof}
  		Using similar arguments in  Theorem~\ref{t:normalizedlp}, we can write
  		\begin{equation}\label{e:complicatedexpression}
  		\begin{split}
  		&\!\!\!\!\!\!\!\!\!V_{n+1}(\lpcurv) \cdot  M_{p\beta}^t(\|V_{A,H}\|_{\infty}^{-1}, \|V_{B,H}\|_{\infty}^{-1})\\
  		&=V_{n+1}(\lpcurv) \cdot \sup_{0 < \lambda < 1} \avgb(\|V_{A,H}\|_{\infty}^{-1}, \|V_{B,H}\|_{\infty}^{-1})\\
  		&\geq \int_{\h} \int_H \sup_{0 < \lambda <1} \sup_{z= C_{p,\lambda,t}x + D_{p,\lambda,t}y} \Phi(x,y) dz,
  		\end{split}
  		\end{equation}
  		where 
  		\[
  		\Phi(x,y) = \avg(V_A(x),V_B(y))\cdot \avgb(\|V_{A,H}\|_{\infty}^{-1}, \|V_{B,H}\|_{\infty}^{-1}).
  		\]
  		Applying (\ref{e:lpholder}) to $\Phi$, we see that 
  		\begin{align*}
  		\Phi(x,y) 
  		\geq \avgc\left(\frac{V_A(x)}{\|V_{A,H}\|_{\infty}},\frac{V_B(y)}{\|V_{B,H}\|_{\infty}} \right)=:\tilde{\Phi}(x,y) 
  		\end{align*}
  		where $\gamma = \frac{\alpha\beta}{\alpha + \beta}$.
  		
  		Write $z = (z_1,z_2) \in H \times \h = \R^n$. Therefore based on \eqref{e:complicatedexpression}, we obtain
  		\begin{align*}
  		&\!\!\!\!\!\!\!\!\!\int_{\h} \int_H \sup_{0 < \lambda <1} \sup_{z= C_{p,\lambda,t}x + D_{p,\lambda,t}y} \tilde{\Phi}(x,y) dz\\
  		& \geq \int_{\h} \sup_{0 < \lambda <1} \sup_{z_2= C_{p,\lambda,t}v + D_{p,\lambda,t}w}
  		\left( \int_{H} \sup_{0 < \lambda <1} \sup_{z_1= C_{p,\lambda,t}u + D_{p,\lambda,t}s} \tilde{\Phi}(u+v,s+w) dz_1 \right) dz_2,
  		\end{align*}
  		where we have 
  			\begin{align*}
  		\tilde{\Phi}(u+v,s+w) &= \avgc\left(\frac{V_A(u+v)}{\|V_{A,H}\|_{\infty}},\frac{V_B( s+w)}{\|V_{B,H}\|_{\infty}} \right)\\
  		&= \avgc\left(\frac{V_1(A \cap (\R_+ +u+v))}{\|V_{A,H}\|_{\infty}},\frac{V_1(B\cap (\R_+ + s+w))}{\|V_{B,H}\|_{\infty}} \right).
  		\end{align*}
  		For each fixed $v,w$ consider the following normalized compression in (\ref{e:criticalquantities}) 
  		\begin{align*}
  		C_H(v) = \left\{(h,r) \colon r \in \left[0, \frac{V_1(A \cap (\R_+ + h+ v))}{\|V_{A,H}\|_{\infty}} \right], A \cap (\R+ h +v) \neq \emptyset \right\},\\
  		D_H(w) = \left\{(h,r) \colon r \in \left[0, \frac{V_1(B \cap (\R+ h+ w))}{\|V_{B,H}\|_{\infty}} \right], B \cap (\R+ h +w) \neq \emptyset \right\}.
  		\end{align*}
  		
  		Notice that $C_H(v) = \widetilde{C_H(v)}$ and $D_H(w) = \widetilde{D_H(w)}$, and so, akin to the proof of Theorem~\ref{t:normalizedlp}, using Fubini's theorem, we have that 
  		\begin{equation}\label{e:integralineq}
  		\begin{split}
  		&\int_{H} \sup_{0 < \lambda <1} \sup_{z_1= C_{p,\lambda,t}u + D_{p,\lambda,t}s} \tilde{\Phi}(u+v,s+w) dz_1\\
  		&=\int_{H} \sup_{0 < \lambda <1} \sup_{z_1= C_{p,\lambda,t}u + D_{p,\lambda,t}s}\left(\avgc\left(\frac{V_1(A \cap (\R_+ +u+v))}{\|V_{A,H}\|_{\infty}},\frac{V_1(B\cap (\R_+ + s+w))}{\|V_{B,H}\|_{\infty}} \right)\right) dz_1\\
  		&=\int_H V_1([ (1-t) \times_{p,\gamma} C_H(v) +_{p,\gamma} t \times_{p,\gamma} D_H(w)] \cap (\R_+ + z_1))dz_1\\
  		&=V_{k+1}((1-t) \times_{p,\gamma} C_H(v) +_{p,\gamma} t \times_{p,\gamma} D_H(w))
  		\end{split}    
  		\end{equation}
  		
  		By applying Theorem~\ref{t:lpcbmfull1} to the last quantity appearing in \eqref{e:integralineq}, 
  		we see that 
  		\begin{align*}
  		&\int_{\h} \sup_{0 < \lambda <1} \sup_{z_2= C_{p,\lambda,t}v + D_{p,\lambda,t}w} \left( \int_{H} \sup_{0 < \lambda <1} \sup_{z_1= C_{p,\lambda,t}u + D_{p,\lambda,t}s} \tilde{\Phi}(u+v,s+w) dz_1 \right) dz_2\\
  		&\geq \int_{\h}\sup_{0 < \lambda <1} \sup_{z_2= C_{p,\lambda,t}v + D_{p,\lambda,t}w}\left[\begin{cases}
  		M_{p,\delta}^{(t,\lambda)}(V_{k+1}(C_H(v)),V_{k+1}(D_H(w))), &\text{if } \gamma \geq -\frac{1}{k},\\
  		\min\left\{\left[C_{p,\lambda,t}\right]^{\frac{1}{\delta}}V_{k+1}(C_H(v)),\left[D_{p,\lambda,t}\right]^{\frac{1}{\delta}} V_{k+1}(D_H(w))\right\}, &\text{if } \gamma < - \frac{1}{k},
  		\end{cases}\right]
  		dz_2,
  		\end{align*}
  		where $\delta = (\alpha^{-1} + \beta^{-1} + k)^{-1}.$
  		
  		Notice that 
  		\begin{align*}
  		&V_{k+1}(C_H(v)) = V_1(A_H \cap (\R_+ + v)),\\
  		&V_{k+1}(D_H(w)) = V_1(B_H  \cap (\R_+ + w)).
  		\end{align*}
  		
  		Finally, noting that $\widetilde{A_H} = A_H$ and $\widetilde{B_H}  = B_H$, similar tactics as before yield the following identities 
  		\begin{align*}
  		&\!\!\!\!\!\!\!V_1([(1-t) \times_{p,\delta} \widetilde{A_H} +_{p,\delta} t \times_{p,\delta} \widetilde{B_H} )]\cap (\R_+ + z_2))\\
  		&= \sup_{0 < \lambda <1} \sup_{z_2= C_{p,\lambda,t}v + D_{p,\lambda,t}w} M_{p,\delta}^{(t,\lambda)}(V_1(\widetilde{A_H} \cap (\R_+ +v)),V_1(\widetilde{B_H}  \cap (\R_+ + w))),\\
  			&\!\!\!\!\!\!\!V_1([(1-t) \times^{p,\delta} \widetilde{A_H} +^{p,\delta} t \times^{p,\delta} \widetilde{B_H}]\cap (\R_+ + z_2))\\
  		&=\sup_{0 < \lambda <1}\sup_{z_2= C_{p,\lambda,t}v + D_{p,\lambda,t}w} R(v,w),
  		\end{align*}
  		where 
  		\[
  		R(v,w) =  \min\left\{\left[C_{p,\lambda,t}\right]^{\frac{1}{\delta}}V_1(\widetilde{A_H} \cap (\R_+ +v)),\left[D_{p,\lambda,t}\right]^{\frac{1}{\delta}} V_1(\widetilde{B_H}  \cap (\R_+ + w))\right\}.
  		\]
  		
  		Putting everything together, by integrating these inequalities over $\h$, and using Fubini's theorem appropriately, we conclude the proof. 
  	\end{proof}

  	\section{ Generalized $L_{p,\bar{\alpha}}$-Borell-Brascamp-Lieb inequality }\label{section4}

  	In \cite{RX, RX2}, the authors proved the $L_p$ Borell-Brascamp-Lieb inequality using the methods of mass transportation, revolution bodies and the classic Borell-Brascamp-Lieb inequality, etc. 	In this section, we extend the $L_p$ Borell-Brascamp-Lieb inequality for functions to  multiple power $L_{p,\bar{\alpha}}$ by proving the  $L_{p,\alpha}$ curvilinear-Brunn-Minkowski inequality for  as well as  utilizing the hypo-graph of functions in Subsection \ref{subsection41}. Based on this inequality, we establish a normalized version of $L_{p,\alpha}$ Borell-Brascamp-Lieb inequality for functions and a normalized  $L_{p,\alpha}$ curvilinear-Brunn-Minkowski inequality for sets in terms of measure in Subsections \ref{subsection42} and \ref{subsection43}, respectively.
  	\subsection{New proof of multiple $L_{p}$-Borell-Brascamp-Lieb inequality for functions}\label{subsection41}

Recall the $L_p$ ($L_{p,\alpha}$) Borell-Brascamp-Lieb inequality for functions in \cite{RX,RX2} below.
  	
  	\begin{theorem}\label{t:stronglpbbl} Let $p \geq 1$, $p^{-1} + q^{-1} = 1$, $t \in (0,1)$, and $\alpha \in [-\infty,+\infty]$. Suppose that $f,g,h \colon \R^n \to \R_+$ are a triple of integrable functions that satisfy the condition 
  		\begin{equation}\label{e:lpbblassumption}
  		h\left(C_{p,\lambda,t}x +D_{p,\lambda,t}y \right) \geq 
  		\avg(f(x),g(y))
  		\end{equation}
  		for all $x,y \in \R^n$ such that $f(x)g(y)>0$ and for every $t, \lambda \in (0,1)$. Then the following integral inequality holds:
  		\begin{equation*}\label{e:lpbblresult}
  		\begin{split}
  		\int_{\R^n}h(x) dx
  		&\geq \begin{cases}
  M_{p\gamma}^t\left(\int_{\R^n}f, \int_{\R^n} g \right) , &\text{if } \alpha \geq -\frac{1}{n},\\
  		\sup_{0 < \lambda < 1} \left\{ \min\left\{\left[C_{p,\lambda,t}\right]^{\frac{1}{\gamma}}\int_{\R^n}f,\left[D_{p,\lambda,t}\right]^{\frac{1}{\gamma}} \int_{\R^n}g\right\} \right\}, &\text{if } \alpha < - \frac{1}{n},
  		\end{cases}
  		\end{split}
  		\end{equation*}
  		where $\gamma = \frac{\alpha}{1+n\alpha}$.
  		
  	\end{theorem}
  	
  	\begin{proof}
  		As $f$ and $g$ are integrable functions, we may assume without loss of generality that $f,g$ are bounded functions supported on a bounded sets. For a bounded integrable function $m \colon \R^n \to \R_+$ with bounded support,  consider its hypo-graph
  		\[
  		\text{hyp}(m) = \{(x,r) \in \R^n \times \R_+ \colon r \in [0, f(x)], x \in \text{supp}(m)\},
  		\]
  		which is a bounded Borel set. It is easy to check that
  		\[
  		\text{hyp}(m) = \widetilde{\text{hyp}(m)}, \quad V_{n+1}(\text{hyp}(m)) = \int_{\R^n}m(x)dx. 
  		\]
  	     
  	     	Observe, for suitable choices of $z \in \R^n$, we have 
  		\begin{align*}
  		&\!\!\!\!\!\!\!\!\![(1-t) \times_{p,\alpha}\text{hyp}(f) +_{p,\alpha} t \times_{p,\alpha}\text{hyp}(g)] \cap (\R_+ + z)\\
  		&=\bigcup_{\lambda \in (0,1)} \bigcup_{z=C_{p,\lambda,t}x + D_{p,\lambda,t}y} \left(z,0
  		\right) \times \left[0, \avg(f(x),g(y) \right],
  		\end{align*}
  	 where $x\in supp(f)$ and $y\in supp(g)$.
  		In particular, the above implies that 
  		\begin{align*}
  		&\!\!\!\!\!\!\!\!\!V_1([(1-t) \times_{p,\alpha}\text{hyp}(f) +_{p,\alpha} t \times_{p,\alpha}\text{hyp}(g)] \cap (\R_+ + z))\\
  		&= \sup_{0 < \lambda < 1} \left[\sup_{z = C_{p,\lambda,t}x + D_{p,\lambda,t}y} \avg(f(x),g(y))\right].
  		\end{align*}
  		Due to the assumption \eqref{e:lpbblassumption}, this shows that 
  		\begin{align*}
  		V_{n+1}(\text{hyp}(h))&=\int_{\R^n}h(z) dz \\&\geq \int_{\R^n} \sup_{0 < \lambda < 1}\left[ \sup_{z = C_{p,\lambda,t}x + D_{p,\lambda,t}y} \avg(f(x),g(y)) \right] dz\\
  		&= \int_{\R^n} V_1([(1-t) \times_{p,\alpha}\text{hyp}(f) +_{p,\alpha} t \times_{p,\alpha}\text{hyp}(g)] \cap (\R_+ + z)) dz\\
  		&= V_{n+1}((1-t) \times_{p,\alpha} \text{hyp}(f) +_{p,\alpha} t \times_{p,\alpha} \text{hyp}(g)).
  		\end{align*}
  		
  		Consequently, we employ the $L_{p,\alpha}$-curvilinear-Brunn-Minkowski inequality for sets, i.e., Theorem~\ref{t:lpcurvfulld} to the sets $\text{hyp}(f)$ and $\text{hyp}(g)$ to obtain
  		\begin{align*}
  	\int_{\R^n}h(x) dx
  		& \geq V_{n+1}((1-t) \times_{p,\alpha} \text{hyp}(f) +_{p,\alpha} t \times_{p,\alpha} \text{hyp}(g))\\
  		&\geq \begin{cases}
  		\sup_{0 < \lambda < 1} \left\{\avgc\left(V_{n+1}(\text{hyp}(f)), V_{n+1}(\text{hyp}(g))\right)\right\}, &\text{if } \alpha \geq -\frac{1}{n},\\
  		\sup_{0 < \lambda < 1} \left\{ \min\left\{\left[C_{p,\lambda,t}\right]^{\frac{1}{\gamma}}V_{n+1}(\text{hyp}(f)),\left[D_{p,\lambda,t}\right]^{\frac{1}{\gamma}} V_{n+1}(\text{hyp}(g)) \right\} \right\}, &\text{if } \alpha < - \frac{1}{n},
  		\end{cases}\\
  		&=\begin{cases}
  		\sup_{0 < \lambda < 1} \left\{\avgc\left(\int_{\R^n}f, \int_{\R^n} g \right) \right\}, &\text{if } \alpha \geq -\frac{1}{n},\\
  		\sup_{0 < \lambda < 1} \left\{ \min\left\{\left[C_{p,\lambda,t}\right]^{\frac{1}{\gamma}}\int_{\R^n}f,\left[D_{p,\lambda,t}\right]^{\frac{1}{\gamma}} \int_{\R^n}g\right\} \right\}, &\text{if } \alpha < - \frac{1}{n},
  		\end{cases}
  		\\
  		&=\begin{cases}
  	M_{p\gamma}^t\left(\int_{\R^n}f, \int_{\R^n} g \right), &\text{if } \alpha \geq -\frac{1}{n},\\
  		\sup_{0 < \lambda < 1} \left\{ \min\left\{\left[C_{p,\lambda,t}\right]^{\frac{1}{\gamma}}\int_{\R^n}f,\left[D_{p,\lambda,t}\right]^{\frac{1}{\gamma}} \int_{\R^n}g\right\} \right\}, &\text{if } \alpha < - \frac{1}{n},
  		\end{cases}
  		\end{align*}	
  		where $\gamma = \frac{\alpha}{1+n\alpha}$, completing the proof. 
  	\end{proof}
 Repeating the proof of Theorem~\ref{t:stronglpbbl}, similar to the process in  Theorem~\ref{t:lpcurvfulld} and  Theorem~\ref{t:lpUhrinBM}, we obtain the following generalization and strengthening of Uhrin's curvilinear Pr\'ekopa-Leindler inequality in \cite{Uhrin2}---the $L_{p,\bar{\alpha}}$-Borell-Brascamp-Lieb inequality.
  	\begin{theorem}\label{t:UhrinPLlp} 
  		Let $p \geq 1$, $p^{-1} + q^{-1} = 1$, $t \in (0,1)$, and $\bar{\alpha}=(\alpha_1,\dots, \alpha_{n+1})$ with $\alpha_i \in [0,1]$ for all $i =1,\dots,n$. Suppose that $f,g,h\colon (\R_+)^n \to \R_+$ are a triple of bounded integrable functions having bounded support that satisfy the condition 
  		\begin{align}\label{multipleconvolution}
  		h\left(M_{p,\alpha_1}^{(t,\lambda)}(x_1,y_1),\dots, M_{p,\alpha_n}^{(t,\lambda)}(x_n,y_n)\right) \geq M_{p,\alpha_{n+1}}^{(t,\lambda)}(f(x_1,\dots,x_n), g(y_1,\dots,y_n))
  		\end{align}
  		for all $x=(x_1,\cdots, x_n), y=(y_1,\cdots, y_n) \in (\R_+)^n$ with $f(x)g(y)=f(x_1,\cdots, x_n)g(y_1,\cdots, y_n) > 0$, and for all $\lambda \in (0,1)$. Then the following inequality holds:
  		\begin{align*}
  		&\int_{(\R_+)^n}h(x) dx\\
  		&\geq \begin{cases}
  		M_{p\gamma}^t\left(\int_{(\R_+)^n}f(x) dx,\int_{(\R_+)^n}g(x) dx \right) ,
  		&\text{if }\alpha_{n+1} \geq -\left(\sum_{i=1}^n \alpha_i^{-1}\right)^{-1},\\
  		\sup_{0< \lambda < 1} \min\left\{\left[C_{p,\lambda,t}\right]^{\frac{1}{\gamma}}\int_{(\R_+)^n}f(x) dx,\left[D_{p,\lambda,t}\right]^{\frac{1}{\gamma}} \int_{(\R_+)^n}g(x) dx\right\},
  		& \text{if } \alpha_{n+1} < -\left(\sum_{i=1}^n \alpha_i^{-1}\right)^{-1}, 
  		\end{cases}
  		\end{align*} 
  		where 
  	$
  		\gamma =\left(\sum_{i=1}^{n+1}\alpha_i^{-1}\right)^{-1}. $
  	\end{theorem}

  	\subsection{Normalized $L_p$-Borell-Brascamp-Lieb type inequality}\label{subsection42}
  	
  	The main focus of this subsection is the following normalized version of the $L_{p,\bar{\alpha}}$  Borell-Brascamp-Lieb inequality to  more general space in previous Subsection. 
  	
  	  	Let $H \subset G_{n,k}$ be a linear subspace for $k\in\{0,1,\cdots,n\}$ and denote by $\h$ its orthogonal complement in $\R^n$. For a integrable  function $h \colon \R^n \to \R_+$, Fubini's theorem implies that
  	\begin{equation*}\label{e:sectionalfubini}
  	\int_{\R^n}h(z) dz = \int_{\h} \left( \int_H h(x+y) dy \right) dx. 
  	\end{equation*}

  	\begin{theorem}\label{t:marginaluhrin} Let $p \geq 1$, $p^{-1}+q^{-1}=1$, $t \in (0,1)$, and $\alpha, \beta \in [-\infty, \infty]$ be such that $\alpha + \beta \geq 0$. Let $H\in G_{n,k}$, $k 
  		\in \{1,\dots,n\}$. Let $f,g,h \colon \R^n \to \R_+$ be a triple of integrable functions, and assume that the hypo-graphs $A = \text{hyp}(f)$ and $B = \text{hyp}(g)$ are bounded $F_{\sigma}$-sets in $\R^n$. Assume that $f,g,h$ satisfy the condition 
  		\begin{equation}\label{e:marginalassumpt}
  		h\left(C_{p,\lambda,t}x+D_{p,\lambda,t}y  \right) \geq \avg(f(x),g(y))
  		\end{equation}
  		for all $x,y \in \R^n$ such that $f(x)g(y) > 0$ and all $\lambda \in (0,1)$.  
  		Consider the norms of marginal functions of $f,g$ with respect to $H\in G_{n,k}$ where $k\in\{0,1,\cdots, n\},$
  		\begin{align*}
  		\|f\|_H = \sup_{z \in \h} \int_{H}f(x+z)dx=:\sup_{z \in \h}I(f,z) , \quad \|g\|_H = \sup_{z \in \h} \int_{H}g(x+z)dx =: \sup_{z \in \h} I(g,z). 
  		\end{align*}
  		Then the following inequality holds: 
  	  		\begin{equation}\label{e:superstronglpbbl}
  	\begin{split}
  	&\left(\int_{\R^n}h(x)dx \right) \cdot  M_{p\beta}^t(\|f\|_H^{-1},\|g\|_H^{-1})\\
  	&\geq  \begin{cases} \int_{\h}
   M_{p\delta}^{t}\left(\frac{I(f,u)}{\|f\|_H}, \frac{I(g,v)}{\|g\|_H} \right) dw
  	&\text{if } \frac{\alpha\beta}{\alpha+\beta} \geq -\frac{1}{k},\\
  	\int_{\h} \sup_{0<\lambda<1} \left\{ \min\left\{\left[C_{p,\lambda,t}\right]^{\frac{1}{\delta}}\frac{I(f,u)}{\|f\|_H},\left[D_{p,\lambda,t}\right]^{\frac{1}{\delta}} \frac{I(g,v)}{\|g\|_H}\right\} \right\}dw,
  	&\text{if } \frac{\alpha\beta}{\alpha+\beta} < -\frac{1}{k}, 
  	\end{cases}
  	\end{split}
  	\end{equation}	
  		where the suprema are taken over all ways to write appropriately selected $w \in \h$ in the form $ w= C_{p,\lambda,t}u + D_{p,\lambda,t}v$, and where $\delta =(\alpha^{-1}+\beta^{-1} +k)^{-1}$.

  	\end{theorem}
  	
  	\begin{proof} The goal of the proof is to invoke Theorem~\ref{t:sectionallpinequality} at the right moment. As we have seen before in the proof of Theorem~\ref{t:stronglpbbl}, 
  		\[
  		\text{hyp}(f) = \widetilde{\text{hyp}(f)}, \quad V_{n+1}(\text{hyp}(f)) = \int_{\R^n}f(x) dx,
  		\]
  		and the same with $g$ and $h$. Additionally, for $z \in \R^n$, we have 
  		\begin{equation}\label{e:equivalence1}
  		\begin{split}
  		&\!\!\!\!\!V_1([(1-t) \times_{p,\alpha}\text{hyp}(f) +_{p,\alpha} t \times_{p,\alpha}\text{hyp}(g)] \cap (\R_+ + z))\\
  		&= \sup_{0 < \lambda < 1} \left[\sup_{z = C_{p,\lambda,t}x + D_{p,\lambda,t}y} \avg(f(x),g(y))\right].
  		\end{split}
  		\end{equation}
  		Integrating \eqref{e:equivalence1} over $z\in\R^n$ and taking into account the condition \eqref{e:marginalassumpt}, we obtain 
  		\begin{align*}
  		\int_{\R^n}h(z) dz &\geq \int_{\R^n} \sup_{0 < \lambda < 1} \left[\sup_{z = C_{p,\lambda,t}x + D_{p,\lambda,t}y} \avg(f(x),g(y))\right] dz\\
  		&= \int_{\R^n}V_1([(1-t) \times_{p,\alpha}\text{hyp}(f) +_{p,\alpha} t \times_{p,\alpha}\text{hyp}(g)] \cap (\R_+ + z)) dz. 
  		\end{align*}
  		The definitions in equation \eqref{e:criticalquantities} yield 
  		\[
  		\|V_{\text{hyp}(f),H} \|_{\infty} = \|f\|_H, \quad \|V_{\text{hyp}(g),H}\| =\|g\|_H,
  		\]
  		and 
  		\begin{align*}
  		&|(1-t) \times^{p,\delta} \text{hyp}(f)_H +^{p,\delta} t \times^{p,\delta} \text{hyp}(g)_H|= \int_{\h} \Psi(w) dw,
  		\end{align*}
  		where 
\[ \Psi(w)=\sup_{0 < \lambda < 1} \left[\sup_{w = C_{p,\lambda,t}u + D_{p,\lambda,t}v} \min\left\{\left(C_{p,\lambda,t} \right)^{\frac{1}{\delta}} \frac{I(f,u)}{\|V_{A,H} \|_{\infty}},\left(D_{p,\lambda,t} \right)^{\frac{1}{\delta}} \frac{I(g,v)}{\|V_{B,H} \|_{\infty}} \right\} \right],
  		\]
  		with $A = \text{hyp}(f)$ and $B = \text{hyp}(g)$.  Finally, by applying Theorem~~\ref{t:sectionallpinequality}, we conclude the proof.
  		
  	\end{proof}
  When $k=n,$ we obtain the normalized $L_p$-Borell-Brascamp-Lieb inequality as follows.
  	\begin{coro}\label{t:marginaluhrin1} Let $p \geq 1$, $p^{-1}+q^{-1}=1$, $t \in (0,1)$, and $\alpha, \beta \in [-\infty, \infty]$ be such that $\alpha + \beta \geq 0$. Let $f,g,h \colon \R^n \to \R_+$ be a triple of integrable functions, and assume that the hypo-graphs $A = \text{hyp}(f)$ and $B = \text{hyp}(g)$ are bounded $F_{\sigma}$-sets in $\R^n$. Assume that $f,g,h$ satisfy the condition 
  	\begin{equation}\label{e:marginalassumpt}
  	h\left(C_{p,\lambda,t}x+D_{p,\lambda,t}y  \right) \geq \avg(f(x),g(y))
  	\end{equation}
  	for all $x,y \in \R^n$ such that $f(x)g(y) > 0$ and all $\lambda \in (0,1)$.  
	Then the following inequality holds: 
  	\begin{equation}\label{e:superstronglpbbl}
  	\begin{split}
  	&\left(\int_{\R^n}h(x)dx \right) \cdot  M_{p\beta}^t(\|f\|_{\infty}^{-1},\|g\|_{\infty}^{-1})\\
  	&\geq  \begin{cases} 
  	M_{p\omega}^{t}\left(\frac{\int_{\R^n}f(x)dx}{\|f\|_{\infty}}, \frac{\int_{\R^n}g(x)dx}{\|g\|_{\infty}} \right) 
  	&\text{if } \frac{\alpha\beta}{\alpha+\beta} \geq -\frac{1}{n},\\
  \sup_{0<\lambda<1} \left\{ \min\left\{\left[C_{p,\lambda,t}\right]^{\frac{1}{\delta}}\frac{\int_{\R^n}f(x)dx}{\|f\|_H},\left[D_{p,\lambda,t}\right]^{\frac{1}{\delta}} \frac{\int_{\R^n}g(x)dx}{\|g\|_H}\right\} \right\},
  	&\text{if } \frac{\alpha\beta}{\alpha+\beta} < -\frac{1}{n}, 
  	\end{cases}
  	\end{split}
  	\end{equation}	
  	where the suprema are taken over all ways to write appropriately selected $w \in \h$ in the form $ w= C_{p,\lambda,t}u + D_{p,\lambda,t}v$, and where $\omega =(\alpha^{-1}+\beta^{-1} +n)^{-1}$.
  \end{coro}	
  	
  	\subsection{Normalized $L_p$-curvilinear-Brunn-Minkowski inequality for measures}\label{subsection43}
  	In this Subsection, we further consider the form for normalized $L_{p,\alpha}$-curvilinear-Brunn-Minkowski inequality for sets in terms of measure.
  	Before stating our next results, we  introduce some preliminary notions. 
Let $\mu$ be a measure on $\R^n$ having a bounded density $\phi \colon \R^n \to \R_+$, and let $A \subset \R^n$ be a compact set of positive $\mu$ measure, and $H\in G_{n,k}$ for some $0\leq k 
\leq n$.  Consider the function 
$
\mu_{A,H}(u) = \mu(A \cap (H-u)), \ u \in \h, 
$
and set 
\[
m_{A,H}^{\mu}: =\sup\{\mu_{A,H}(u): u \in \h\}.
\]
We consider the normalized super-level sets of $\mu_{A,H}$ given for $r\in[0,1]$
\begin{equation*}\label{e:levelsets}
C_{A,H}^{\mu}(r) = \left\{u \in \h \colon \mu_{A,H}(u) \geq r \cdot m_{A,H}^{\mu} \right\}\subset H^{\perp}\in G_{n,n-k}.
\end{equation*}
In particular, using Fubini's theorem, we see that
\begin{equation}\label{e:nicefubini}
\begin{split}
\mu(A) = \int_{\R^n} (\phi \chi_A)(z) dz
= \int_{\h} \int_H (\phi \chi_A)(u+w)dwdu
=\int_{\h} \mu_{A,H}(u) du
= m_{A,H}\int_0^1 \mu(C_{A,H}^{\mu}(r))dr.
\end{split}
\end{equation}

\begin{theorem}\label{t:marginallpbmi} Let $p \geq 1$, $p^{-1} + q^{-1} = 1$, $t \in (0,1)$, and $\alpha,\beta \in (-\infty,+\infty)$, with $\alpha+\beta \geq 0$, and $H\in G_{n,k}$  for $k\in\{0,\cdots,n\}$. Let $\mu$ be a measure on $\R^n$ given by $d\mu(x) = \phi(x) dx$, where $\phi \colon \R^n \to \R_+$ is bounded $\alpha$-concave function whose support contains the origin. For any compact sets $A,B \subset \R^n$, each having positive $\mu$-measure, one has 
	\begin{equation}\label{e:marginallpbmi0}
	\begin{split}
	&\!\!\!\! \mu((1-t) \cdot_p A +_p t \cdot_p B) \cdot M_{p\beta}^t \left(m_{A,H}^{-1},m_{B,H}^{-1}\right)\\
	&\geq   \begin{cases} \int_0^1 \mu((1-t) \cdot_p C_{A,H}(r) +_p t \cdot_p C_{B,H}(r))dr, &\text{if } \frac{\alpha\beta}{\alpha+\beta} \geq -\frac{1}{k},\\
	\int_{\h} \sup \left\{ \min\left\{\left[C_{p,\lambda,t}\right]^{\frac{1}{\delta}}\frac{\mu_{A,H}(u)}{m_{A,H}},\left[D_{p,\lambda,t}\right]^{\frac{1}{\delta}} \frac{\mu_{B,H}(v)}{m_{B,H}}\right\} \right\}dw, &\text{if } \frac{\alpha\beta}{\alpha+\beta} < -\frac{1}{k}, 
	\end{cases}
	\end{split}
	\end{equation}
	where the suprema are taken over all $\lambda \in (0,1)$ and all ways to write $w = C_{p,\lambda,t}u + D_{p,\lambda,t}v$ with $u,v \in \h$, and where
	$\delta =(\alpha^{-1}+\beta^{-1} +k)^{-1}$. 
\end{theorem}

  	\begin{proof}
  		
  		Let $\phi \colon \R^n \to \R_+$ be the density of the measure $\mu$. Consider the functions $f = \phi \chi_A$, $g = \phi \chi_B$, and $h = \phi \chi_{(1-t) \cdot_p A+_p t \cdot_p B}.$  Then,  for $\lambda, t\in (0,1)$,  using H\"older's inequality together with the $\alpha$-concavity of $\phi$ and the fact that the support of $\phi$ contains the origin, we may write 
  		\begin{align*}
  		h\left(C_{p,\lambda,t}x+D_{p,\lambda,t}y  \right)  &=\phi \chi_{(1-t) \cdot_p A+_p t \cdot_p B} \left(C_{p,\lambda,t}x+D_{p,\lambda,t}y  \right)\\
  		&= \phi\left(C_{p,\lambda,t}x+D_{p,\lambda,t}y  \right)\\
  		&\geq \avg(f(x),g(y)),
  		\end{align*}
  		holds for any $x\in A, y\in B$, and so the triple $f,g,h$ satisfy the assumptions \eqref{e:marginalassumpt} of Theorem~\ref{t:marginaluhrin}. 
  		
  		Moreover, we see that 
  		\begin{align*}
  		\|f\|_H = \sup_{u \in \h} \int_H f(x+u)du
  		= \sup_{x \in \h} \int_H(\phi \chi_A)(x+u) du
  		= \sup_{u \in \h} \mu(A \cap (H-u)) = m_{A,H},
  		\end{align*}
  		and also that, for each $u \in \h$,  the marginal $I(f,u) = \mu(A \cap (H - u)) = \mu_{A,H}(u)$. Similarly, $\|g\|_H = m_{B,H}$, and for any $v \in \h$, we have $I(g,u) =\mu(B \cap (H - v)) = \mu_{B,H}(v)$. 
  		
  		Suppose that $\beta \in (-\infty,\infty)$ is such that $\alpha+\beta \geq 0$. By applying inequality \eqref{e:superstronglpbbl} to the triple $f,g,h,$ we obtain
  		\begin{equation}\label{e:verywell}
  		\begin{split}
  		&\!\!\!\!\!\!\!\!\!\mu((1-t) \cdot_p A +_p t \cdot_p B) \cdot \sup_{0 < \lambda < 1} \avgb(m_{A,H}^{-1},m_{B,H}^{-1})\\
  		&=  \int_{\R^n}h(z) dz\cdot \sup_{0 < \lambda < 1} \avgb(\|f\|_H^{-1},\|g\|_H^{-1})\\
  		&\geq   \begin{cases} \int_{\h}
  		\sup\left\{ M_{p, \delta}^{(t,\lambda)}\left(\frac{I(f,u)}{\|f\|_H}, \frac{I(g,v)}{\|g\|_H} \right) \right\}  dw,
  		&\text{if } \frac{\alpha\beta}{\alpha+\beta} \geq -\frac{1}{k},\\
  		\int_{\h} \sup \left\{ \min\left\{\left[C_{p,\lambda,t}\right]^{\frac{1}{\delta}}\frac{I(f,u)}{\|f\|_H},\left[D_{p,\lambda,t}\right]^{\frac{1}{\delta}} \frac{I(g,v)}{\|g\|_H}\right\} \right\}dw,
  		&\text{if } \frac{\alpha\beta}{\alpha+\beta} < -\frac{1}{k}, 
  		\end{cases}\\
  		&=   \begin{cases} \int_{\h}
  		\sup\left\{ M_{p, \delta}^{(t,\lambda)}\left(\frac{\mu_{A,H}(u)}{m_{A,H}}, \frac{\mu_{B,H}(v)}{m_{B,H}} \right) \right\}  dw,
  		&\text{if } \frac{\alpha\beta}{\alpha+\beta} \geq -\frac{1}{k},\\
  		\int_{\h} \sup \left\{ \min\left\{\left[C_{p,\lambda,t}\right]^{\frac{1}{\delta}}\frac{\mu_{A,H}(u)}{m_{A,H}},\left[D_{p,\lambda,t}\right]^{\frac{1}{\delta}} \frac{\mu_{B,H}(v)}{m_{B,H}}\right\} \right\}dw,
  		&\text{if } \frac{\alpha\beta}{\alpha+\beta} < -\frac{1}{k}, 
  		\end{cases}
  		\end{split}
  		\end{equation}
  		where the suprema are taken over all $\lambda \in (0,1)$ and over all ways to write appropriately selected $w \in \h$ in the form $ w= C_{p,\lambda,t}u + D_{p,\lambda,t}v$ where $\delta =(\alpha^{-1}+\beta^{-1} +k)^{-1}$. 
  		
  		For the case when $\frac{\alpha\beta}{\alpha + \beta} \geq - \frac{1}{k}$, we claim that 
  		\begin{align*}
  		&\!\!\!\!\!\!\!\!\!\!\int_{\h}
  		\sup\left\{ M_{p, \delta}^{(t,\lambda)}\left(\frac{\mu_{A,H}(u)}{m_{A,H}}, \frac{\mu_{B,H}(v)}{m_{B,H}} \right) \colon  w= C_{p,\lambda,t}u + D_{p,\lambda,t}v, 0 < \lambda <1 \right\}  dw\\
  		&\geq \int_0^1 \mu((1-t) \cdot_p C_{A,H}(r) +_p t \cdot_p C_{B,H}(r)) dr. 
  		\end{align*}
  		Indeed, if we  consider the function $M \colon \h \to \R_+$ defined by 
  		\[
  		M(w) = \sup\left\{ M_{p, \delta}^{(t,\lambda)}\left(\frac{\mu_{A,H}(u)}{m_{A,H}}, \frac{\mu_{B,H}(v)}{m_{B,H}} \right)\colon w= C_{p,\lambda,t}u + D_{p,\lambda,t}v, 0 < \lambda <1  \right\},
  		\]
  		then it is not difficult to see that 
  		\begin{align*}
  		\{w \in \h \colon M(w) \geq r \} &\supset (1-t) \cdot_p \left\{u \in \h \colon \frac{\mu_{A,H}(u)}{m_{A,H}} \geq r \right\} +_p t \cdot_p \left\{v \in \h \colon \frac{\mu_{B,H}(v)}{m_{B,H}} \geq r \right\}\\
  		&= (1-t) \cdot_p C_{A,H}(r) +_p t \cdot_p C_{B,H}(r).
  		\end{align*}
  		
  		Therefore, for any $\beta$ such that $\alpha+\beta \geq 0$ and $\frac{\alpha\beta}{\alpha + \beta} \geq - \frac{1}{k}$, using the inequality \eqref{e:verywell} together with Fubini's theorem and the above inclusion, we obtain 
  		\begin{align*}
  		&\!\!\!\!\!\!\!\mu((1-t) \cdot_p A +_p t \cdot_p B) \cdot M_{p\beta}^t \left(m_{A,H}^{-1},m_{B,H}^{-1}\right)\\
  		&=\mu((1-t) \cdot_p A +_p t \cdot_p B) \cdot  \sup_{0 < \lambda < 1} \avgb(m_{A,H}^{-1},m_{B,H}^{-1})\\
  		& \geq \int_{\h}M(w) dw\\
  		&= \int_0^1 \mu(\{w \in \h \colon M(w) \geq r \})dr\\
  		&\geq \int_0^1\mu((1-t) \cdot_p C_{A,H}(r) +_p t \cdot_p C_{B,H}(r)) dr,
  		\end{align*}
  		which verifies Theorem~\ref{t:marginallpbmi} in the case that $\alpha+\beta \geq 0$ and $\frac{\alpha\beta}{\alpha + \beta} \geq - \frac{1}{k}$. The case for $\frac{\alpha\beta}{\alpha + \beta}<- \frac{1}{k}$ can be proved in a similar way as in Theorem \ref{t:marginaluhrin}.
  	\end{proof}
  By choosing $k = n$ in Theorem~\ref{t:marginallpbmi}, the right-hand side of inequality \eqref{e:marginallpbmi0} satisfies 
  \[
  \int_0^1 \mu((1-t) \cdot_p C_{A,H}(r) +_p t \cdot_p C_{B,H}(r))dr \geq \sup_{0 < \lambda <1} M_{p,1}^{(t,\lambda)}\left(\frac{\mu(A)}{m_{A,H}}, \frac{\mu(B)}{m_{B,H}} \right).
  \]
  Therefore, using the identity \eqref{e:nicefubini}, we obtain the following corollary. 
  
  \begin{coro}
  	Let $p \geq 1$, $p^{-1} + q^{-1} = 1$, $t \in (0,1)$, and  $H\in G_{n,k}$ for some $k\in\{0,\cdots,n\}$. Let $\alpha,\beta \in (-\infty,\infty)$ be such that $\alpha + \beta \geq 0$ and $\alpha\beta/(\alpha+\beta) \geq -1/k$. Let $\mu$ be a measure on $\R^n$ whose density is a bounded $\alpha$-concave function whose support contains the origin. For any compact sets $A,B \subset \R^n$, each having positive $\mu$-measure, one has 
  	\begin{equation*}\label{e:bonnesson}
  	\begin{split}
  	&\!\!\!\!\!\!\!\mu((1-t) \cdot_p A +_p t \cdot_p B) \cdot \sup_{0 < \lambda < 1} \avgb\left(m_{A,H}^{-1},m_{B,H}^{-1}\right)\\
  	&\geq \sup_{0 < \lambda <1} M_{p,1}^{(t,\lambda)}\left(\frac{\mu(A)}{m_{A,H}}, \frac{\mu(B)}{m_{B,H}} \right).
  	\end{split}
  	\end{equation*}
  	In particular, choosing $p=1$, we obtain 
  	\begin{align*}
  	\mu((1-t) A + t B) &\geq \left[(1-t) (m_{A,H})^{-\beta} + t (m_{B,H})^{-\beta} \right]^{-\frac{1}{\beta}} \cdot \left[(1-t) \frac{\mu(A)}{m_{A,H}} + t \frac{\mu(B)}{m_{B,H}} \right]. 
  	\end{align*}
  \end{coro}

  	\section{$L_{p, \bar{\alpha}}$-Minkowski first type and  isoperimetric inequalities for measures }\label{section5}

 Recall the definition for $L_p$ mixed volume for convex bodies in (\ref{variationconvexbodyp}) derived from the variation formula for volume (Lebesgue measure) in terms of the $L_p$ Minkowski summation for convex bodies. In Subsection \ref{subsection51}, we will define the surface area for measures with respect to the multiple $L_p$ curvilinear summation for Borel sets. 
  Additionally, we will examine the corresponding $L_p$ versions of the Minkowski's first inequality and isoperimetric inequality, etc. Based on the $L_{p,\bar{\alpha}}$ Borell-Brascamp-Lieb inequality, we further propose the definition of multiple supremal-convolution for functions together with the $L_p$  Minkowski's first inequality and isoperimetric inequality with respect to this new convolution definition.
  	
  	  	\subsection{$L_{p, \bar{\alpha}}$-Minkowski first type and  isoperimetric inequalities for sets}\label{subsection51}
  	Firstly, we pose the concept of $F$-concavity with respect to the  $L_{p, \bar{\alpha}}$ curvilinear summation for sets defined in Section \ref{section2}.
  
  	\begin{defi}\label{definitionoffconcavity} Let $p \in [1,\infty)$, $\bar{\alpha}=(\alpha_1,\cdots,\alpha_{n+1})\in [0,\infty]^{n+1}$, and $F \colon \R_+ \to \R$ be an invertible differentiable function.  We say that a measure $\mu$ on $\R^{n+1}$ is $F(t)$-concave with respect to the $L_{p,\bar{\alpha}}$-curvilinear summation of  bounded Borel sets $A, B\subset\R^{n+1}$ and any $t \in [0,1]$, if
  		\begin{equation}\label{e:functionalMink1stassumption}
  		\mu((1-t) \times_{p,\bar{\alpha}} A +_{p,\bar{\alpha}} t \times_{p,\bar{\alpha}} B)\geq F^{-1}((1-t) F(\mu(A)) + t F(\mu(B)))
  		\end{equation}
  	\end{defi}
\noindent Particularly, if  $\bar{\alpha}=(1,\cdots, 1, \alpha)$ in $n+1$-dimensional space, we call the measure $\mu$ above is $F(t)$-concave with respect to the $L_{p,\alpha}$-curvilinear summation of  bounded Borel sets $A, B.$ Furthermore, when $\alpha=1$, it recovers the definition for \cite{Wu} $F(t)$-concave of $L_p$ Minkowski summation for sets in $\R^{n+1}$.
  Inspired by the works in \cite{ColesantiFragala, Klartag, RX}, we introduce the surface area extension of $L_{p,\bar{\alpha}}$-curvilinear summation  for measures below.
  	
  	\begin{defi}Let $\mu$ be a Borel measure on $\R^{n+1}$, $p \geq 1$, and $\bar{\alpha}=(\alpha_1,\cdots,\alpha_{n+1})\in [0,\infty]^{n+1}$. We define the $L_{p,\bar{\alpha}}$-$\mu$-surface area of a $\mu$-integrable set $A$ with respect to a $\mu$-integrable set $B$ by 
  		\[
  		S_{\mu,p, \bar{\alpha}}(A,B) := \liminf_{\e \to 0^+} \frac{	\mu(A +_{p,\bar{\alpha}} \e \times_{p,\bar{\alpha}} B) -\mu(A)}{\e}.
  		\]
  		If $\mu$ is the Lebesgue measure on $\R^{n+1}$, we will simply denote it as $S_{p,\bar{\alpha}}$.
  	\end{defi}
  	
 Next we will establish the $L_p$-Minkowski's first inequality in terms of the $L_{p,\bar{\alpha}}$-$\mu$-surface area for $p\geq1$ below.
  	
  	\begin{theorem}\label{t:functionalMink1st} Let $p \in[1,\infty)$, $\bar{\alpha}=(\alpha_1,\cdots,\alpha_{n+1})\in [0,\infty]^{n+1}$ and $F \colon \R_+\to \R$ be a differentiable invertible function. Let $\mu$ be a Borel measure on $\R^{n+1}$, and assume that $\mu$ is $F(t)$-concave of $L_{p,\bar{\alpha}}$-curvilinear summation for  any two  $\mu$-measurable sets $A,B\subset\R^{n+1}$.
  		Then the following inequality holds: 
  		\begin{equation}\label{e:functionalMink1stconclusion}
  		S_{\mu,p, \bar{\alpha}}(A,B) \geq S_{\mu,p, \bar{\alpha}}(A,A) + \frac{F\left(\mu(B)\right) - F\left(\mu(A)\right)}{F'\left(\mu(A) \right)}.
  		\end{equation}
  		When $\mu(A) = \mu(B)$,  the following isoperimetric type inequality holds: 
  		\[
  		S_{\mu,p, \bar{\alpha}}(A,B) \geq S_{\mu,p, \bar{\alpha}}(A,A). 
  		\]
  	\end{theorem}

  	\begin{proof}
  		
  	Consider  first that $\alpha_i \neq 0, +\infty$ for all $1\leq i\leq n+1$. It follows from  \eqref{e:functionalMink1stassumption} and the assumption for measure $\mu$ being $F(t)$-concave of $L_{p,\bar{\alpha}}$-curvilinear summation for  any two  $\mu$-measurable sets, for any $\e>0$ sufficiently small, that
  		\begin{eqnarray*}
  		\mu(A \oplus_{p,\bar{\alpha}} (\e \times_{p,\bar{\alpha}} B))  
  		&=&\mu\Big(\left[(1-\e) \times_{p,\bar{\alpha}} \left(\frac{1}{1-\e} \times_{p,\bar{\alpha}}A\right)\right] \oplus_{p,\bar{\alpha}} (\e \times_{p,\bar{\alpha}} B)\Big)\\
  		&\geq& F^{-1}\left\{(1-\e) F\left(\mu(\frac{1}{1-\e} \times_{p,\bar{\alpha}}A) \right) + \e F\left( \mu(B) \right)\right\}\label{concaveinequality}.
  		\end{eqnarray*}
  	For convenience, we denote the following function with parameter $\e$ as
  		\begin{equation}\label{inversesumofF}
  		Q_{F, \mu,s,p}(\e) := F^{-1}\left[(1-\e) F\left(\mu(\frac{1}{1-\e} \times_{p,\bar{\alpha}}A) \right) + \e F\left( \mu(B)\right)\right].
  		\end{equation}
  	It is easy to check that  the equality in (\ref{concaveinequality}) holds true when  $\e = 0,$ i.e., $Q_{F,\mu,s,p}(0) = \mu(A)$. Furthermore, we obtain the relation between the $L_{p,\bar{\alpha}}$-$\mu$-surface area and the function $Q_{F,\mu,s,p}$,
  		\begin{align*}
  		S_{\mu,p, \bar{\alpha}}(A,B) &=\liminf_{\e \to 0^+} \frac{	\mu(A +_{p,\bar{\alpha}} t \times_{p,\bar{\alpha}} B) -\mu(A)}{\e}\\
  		&\geq \liminf_{\e \to 0^+} \frac{Q_{F,\mu,s,p}(\e) - Q_{F,\mu,s,p}(0)}{\e} = Q_{F,\mu,s,p}'(0).
  		\end{align*}
Observe that 
  		\begin{align*}
  		\frac{d}{d\e}\left[\mu(\frac{1}{1-\e} \times_{p,\bar{\alpha}} A) \right] \biggr\rvert_{\e=0}
  		&= \lim_{\e \to 0^+} \frac{\mu( [(1+\e+\e^2+\cdots) \times_{p,\bar{\alpha}} A])- \mu(A)}{\e}
  		= S_{\mu,p, \bar{\alpha}}(A,A). 
  		\end{align*}
Therefore, it follows from 	 (\ref{inversesumofF}) that
  		\begin{align*}
  		S_{\mu,p, \bar{\alpha}}(A,B) \geq G_{F,\mu,s,p}'(0)
  		&= \frac{(1-\e) S_{\mu,p, \bar{\alpha}}(A,A)  \cdot F'\left(\mu(\frac{1}{1-\e} \times_{p,\bar{\alpha}} A )\right) \mid_{\e=0} }{F'\left(\mu(A) \right)}\\
  		&\quad+\frac{- F \left(\mu([(1-\e) \times_{p,\bar{\alpha}} A]) \right)\mid_{\e=0}+F\left(\mu(B)\right)}{F'\left(\mu(A) \right)}\\
  		&=S_{\mu,p, \bar{\alpha}}(A,A) + \frac{F\left(\mu(B)\right) -F\left(\mu(A) \right)}{F'\left(\mu(A)\right)},
  		\end{align*}
  		as desired.
  		
  		The proofs when $\alpha_i=0, \infty$ follow in similar procedures and thus omitted.
  	\end{proof}
  	In particular, we obtain a natural  corollary of Theorem~\ref{t:functionalMink1st} for Lebesgue measure $\mu=V_{n+1}(\cdot),$ then the $L_p$ surface area with respect to the $L_{p,\bar{\alpha}}$-curvilinear summation has the following form,
  		\begin{align*}
  	S_{p, \bar{\alpha}}(A,B) &=\lim_{\e \to 0^+} \frac{	V_{n+1}(A +_{p,\bar{\alpha}} \e\times_{p,\bar{\alpha}} B) -V_{n+1}(A)}{\e}.
  	\end{align*}
  	Moreover, we have its corresponding Minkowski's type inequality and isoperimetric inequality accordingly by Theorem \ref{t:lpUhrinBM}.
  	\begin{coro} 
  		Let $p \in [1,\infty)$ and $\bar{\alpha}=(\alpha_1,\cdots,,\alpha_{n+1})\in (0,1]^{n+1}$. Then, for any bounded measurable sets $A, B\subset \R^{n+1}$, one has 
  		\[
  		S_{p,\bar{\alpha}}(A,B) \geq S_{p,\bar{\alpha}}(A,A) +\frac{V_{n+1}(B)^{p\gamma}-V_{n+1}(A)^{p\gamma}}{p\gamma V_{n+1}(A)^{p\gamma-1}},
  		\]
  		where $\gamma =\left(\sum_{i=1}^{n+1}\alpha_i^{-1}\right)^{-1}.$
  		In particular, when $V_{n+1}(A) = V_{n+1}(B)>0$, we obtain the following isoperimetric type inequality:
  		\[
  		S_{p,\bar{\alpha}}(A,B) \geq S_{p,\bar{\alpha}}(A,A).
  		\] 	
  	\end{coro}
  	
Following the ideas in \cite{RX,RX2}, except for the case that the measure $\mu$ can be $F(t)$-concave of $L_{p,\bar{\alpha}}$-curvilinear summation for  any two  $\mu$-measurable sets, 
we propose  the composite of the function $F$ and measure $\mu$---$F\circ\mu$ in terms of the $L_{p,\bar{\alpha}}$-curvilinear summation: given a measure $\mu$ on $\R^n$ which is $F(t)$-concave in terms of Bounded $A,B \subset \R^n$,  we define the $L_p$-$\mu$-surface area of the set $A$ with respect to the set $B$---$\mathbb{V}_{p,F}^{\bar{\alpha}, \mu}(A,B)$ as
  	\begin{equation}\label{mixedvolume}
  	\mathbb{V}_{p,F}^{\bar{\alpha}, \mu}(A,B) := F'(1) \cdot \liminf_{\e\to 0^+} \frac{\mu(A+_{p,\bar{\alpha}}\e\times_{p,\bar{\alpha}}B) - \mu(A)}{\e}, \end{equation}
  	and denote by
  	\begin{equation}\label{surfacearea}
  	\mathbb{M}_{p,F}^{\bar{\alpha}, \mu}(A):= \frac{1}{F'(1)}\cdot \mu(A) -\frac{d}{d\e}^{-} \biggr\rvert_{\e=1} \mu(\e \times_{p,\bar{\alpha}} A).
  	\end{equation} 
  	We futher need the following concavity property of $F\circ\mu$ with respect of the $L_p$-Minkowski convex combination.
  	\begin{lemma}\label{t:extragoodcoro}
  		Let $p \geq 1$, $F \colon \R_+ \to \R$ be a strictly increasing differentiable function and $\mu$ be a measure on $\R^n$ that is absolutely continuous with respect to the Lebesgue measure. Suppose that  $F\circ \mu$ is concave with bounded Borel sets in terms of $L_{p,\bar{\alpha}}$-curvilinear summation.  Then, for any bounded Borel sets $A,B$, the maps  
  		\[
  		\e \mapsto F(\mu(\e \times_{p,\bar{\alpha}} A)), \quad \e \mapsto F(\mu(A+_{p,\bar{\alpha}} \e \times_{p,\bar{\alpha}} B))
  		\]
  		are concave on $[0,\infty).$
  	\end{lemma}
  	
  	\begin{proof} Let $t \in [0,1]$ and $\e_1,\e_2 \geq 0$.
  	 Therefore, by definition of $L_p$-curvilinear summation, 
  		\[
  		A+_{p,\bar{\alpha}} ((1-t)\e_1 + t \e_2) \times_{p,\bar{\alpha}} B = (1-t) \times_{p,\bar{\alpha}} [A+_{p,\bar{\alpha}} \e_1 \times_{p,\bar{\alpha}} B] +t \times_{p,\bar{\alpha}}[A + \e_2 \times_{p,\bar{\alpha}} B].
  		\]
  	It follows from the fact  $F\circ\mu$ is concave with respect to bounded sets in terms of $L_{p,\bar{\alpha}}$-curvilinear summation that
  		\begin{align*}
  		&\!\!\!\!\!\!\!\!\!\!F(\mu(A+_{p,\bar{\alpha}} ((1-t)\e_1 + t \e_2) \times_{p,\bar{\alpha}} B)) \\
  		&= F(\mu((1-t) \times_{p,\bar{\alpha}} [A+_{p,\bar{\alpha}} \e_1 \times_{p,\bar{\alpha}} B] +t \times_{p,\bar{\alpha}}[A + \e_2 \times_{p,\bar{\alpha}} B]))\\
  		&\geq (1-t) F(\mu(A+_{p,\bar{\alpha}} \e_1 \times_{p,\bar{\alpha}} B))+ t F(\mu(A+_{p,\bar{\alpha}} \e_2 \times_{p,\bar{\alpha}} B)),
  		\end{align*}
  		as desired.
  	The proof of the second inequality assertion follows in similar lines. 
  	\end{proof}
  	
  	Together with the definition of $L_p$-$\mu$-surface area (\ref{mixedvolume}), (\ref{surfacearea}) and Lemma \ref{t:extragoodcoro}, we are now prepared to establish the following isoperimetric type inequality. 
  	
  	\begin{theorem}\label{t:lpisoperimetricinequalitygeneral}
  		Let $p \geq 1$, $F \colon \R_+ \to \R$ be a strictly increasing differentiable function and $\mu$ be a measure on $\R^n$ that is absolutely continuous with respect to the Lebesgue measure. Suppose that $\mu$ is $F(t)$-concave with respect bounded Borel sets.  Then for any bounded Borel sets $A,B$, one has 
  		\[
  		\mathbb{V}_{p,F}^{\bar{\alpha}, \mu}(A,B) + F'(1) \mathbb{M}_{p,F}^{\bar{\alpha}, \mu}(A) \geq \frac{F'(1)[F(\mu(B))-F(\mu(A))]}{F'(\mu(A))}+\mu(A),
  		\]
  		with equality only if and only if $A=B$.
  	\end{theorem}
  	
  	\begin{proof} 

  Denote the function $f \colon [0,1] \to \R_+$  as
  		\[
  		f(t) = F(\mu((1-t) \times_{p,\bar{\alpha}} A +_{p,\bar{\alpha}} t \times_{p,\bar{\alpha}} B))-\left[(1-t) F(\mu(A)) + t F(\mu(B))\right].
  		\]
  		As $F\circ\mu$ is concave with respect the $L_{p,\bar{\alpha}}$-curvilinear summation for bounded Borel sets, the functions $f$ is concave  and such that $f(0) = f(1) =0,$ and the right-derivative of $f$ at $t=0$ exists satisfying
  		\[
  		\frac{d^+}{dt}\biggr\rvert_{t=0}f(t) \geq 0, 
  		\]
  		with equality if and only if $f(t) = 0$ for all $t \in [0,1]$.
  		
We only need to compute the right derivative at $0$ of $\mu((1-t)\times_{p,\bar{\alpha}} A +_{p,\bar{\alpha}} t \times_{p,\bar{\alpha}} B)$ since
  		\[
  		\frac{d^+}{dt}\biggr\rvert_{t=0}f(t) =F'(\mu(A))\cdot \frac{d^+}{dt}\biggr\rvert_{t=0}\mu((1-t) \times_{p,\bar{\alpha}}A + t \times_{p,\bar{\alpha}} B) + F(\mu(A))-F(\mu(B)).
  		\]
  	
  		 Setting $g(r,s)=\mu(r \times_{p,\bar{\alpha}}(A+_{p,\bar{\alpha}} s \times_{p,\bar{\alpha}} B))$, we obtain
  		\begin{align*}
  		\frac{d^+}{dt}\biggr\rvert_{t=0}\mu((1-t) \times_{p,\bar{\alpha}}A + t \times_{p,\bar{\alpha}} B) &= \frac{d^+}{dt}\biggr\rvert_{t=0} g\left(1-t, \frac{t}{1-t}\right)\\
  		&=-\frac{d^-}{dt}\biggr\rvert_{t=1}\mu(t \times_{p,\bar{\alpha}} A) + \frac{d^+}{dt}\biggr\rvert_{t=0}\mu(A+_{p,\bar{\alpha}} t \times_{p,\bar{\alpha}} B)\\
  		&= \mathbb{M}_{p,F}^{\bar{\alpha}, \mu}(A) - \frac{1}{F'(1)}\mu(A) + \frac{1}{F'(1)}\mathbb{V}_{p,F}^{\bar{\alpha}, \mu}(A,B)
  		\end{align*}
  		and then
  		\[
  		\frac{d^+}{dt}\biggr\rvert_{t=0}f(t) = F'(\mu(A))\left[\mathbb{M}_{p,F}^{\bar{\alpha}, \mu}(A)- \frac{1}{F'(1)}\mu(A) + \frac{1}{F'(1)}V_{p}^{\mu}(A,B)\right]+ F(\mu(A))-F(\mu(B)).
  		\]
It follows from the fact  $\frac{d^+}{dt}|_{t=0}f(t) \geq 0$ and the equality case of $\frac{d^+}{dt}\rvert_{t=0}f(t) = 0$ the desired inequality holds true.
  	\end{proof}

  	\subsection{$L_{p, \bar{\alpha}}$ supremal-convolution and corresponding inequalities for functions}\label{subsection52}
  		Based on the condition for $L_{p,\bar{\alpha}}$ Borell-Brascamp-Lieb inequality (\ref{multipleconvolution}), similar to the definition of $L_{p,s}$ supremal-convolution (\ref{e:newsupcolvolution}) originated from the classic Borell-Brascamp-Lieb inequality, we define a generalized version of supremal-convolution for multiple power parameters in $L_p$ case below.
  	\begin{defi}
  		Let $p \geq 1$, $p^{-1} + q^{-1} = 1$, $t \in (0,1)$, and  $\bar{\alpha}=(\alpha_1,\dots, \alpha_{n+1})$ with  $\alpha_i \in [-\infty,\infty]$ for all $i =1,\dots,n+1$. Suppose that $f,g\colon (\R_+)^n \to \R_+$ are a triple of bounded integrable functions having bounded support. We define the $L_{p, \bar{\alpha}}$ supremal-convolution of $f$ and $g$ as 
  		\begin{align*}
  	&\!\!\!\!\!\!\!\!\!\!	((1-t)\times_{p,\bar{\alpha}}f+_{p,\bar{\alpha}}t\times_{p,\bar{\alpha}}g)
  		(z_1,\cdots, z_n)\\	
  		&=\sup_{0<\lambda<1}\left( \sup_{z_i=M_{p,\alpha_i}^{(t,\lambda)}(x_i,y_i), 1\leq i\leq n} M_{p,\alpha_{n+1}}^{(t,\lambda)}(f(x_1,\cdots, x_n), g(y_1,\cdots, y_n))\right).
  		\end{align*}
  	\end{defi}
  	
  	\noindent If $\alpha_i=1$ for $1\leq i\leq n$ and $\alpha_{n+1}=\frac{1}{s}$, it recovers the definition of $L_{p,\bar{\alpha}}$-supremal-convolution in \cite{RX}.
  	
  	By the definitions of $L_{p,\bar{\alpha}}$-supremal-convolution, we conclude that with the same conditions for $\bar{\alpha}$, $\gamma$, $f$ and $g$ in Theorem \ref{t:UhrinPLlp},  one has the following Brunn-Minkowski type inequality
  	\begin{align*}
  	&	\int_{\R^n}\big((1-t)\times_{p,\bar{\alpha}}f\oplus_{p,\bar{\alpha}}t\times_{p,\bar{\alpha}}g\big)(x)dx\\
  	&\geq \begin{cases}
  	M_{p\gamma}^t\left(\int_{(\R_+)^n}f(x) dx,\int_{(\R_+)^n}g(x) dx \right) ,
  	&\text{if }\alpha_{n+1} \geq -\left(\sum_{i=1}^n \alpha_i^{-1}\right)^{-1},\\
  	\sup_{0< \lambda < 1} \min\left\{\left[C_{p,\lambda,t}\right]^{\frac{1}{\gamma}}\int_{(\R_+)^n}f(x) dx,\left[D_{p,\lambda,t}\right]^{\frac{1}{\gamma}} \int_{(\R_+)^n}g(x) dx\right\},
  	& \text{if } \alpha_{n+1} < -\left(\sum_{i=1}^n \alpha_i^{-1}\right)^{-1}.
  	\end{cases}
  	\end{align*}  	
  	
  	Next, similar to the definition of surface area defined in above subsection in terms of $L_{p, \bar{\alpha}}$-curvilinear summation for sets with respect to measures, we will give the surface area for functions with respect to the $L_{p,\bar{\alpha}}$--supremal convolution for functions in this section accordingly, which is also a natural multiple generalization in \cite{RX}.

We also have a $F$-concavity in terms of the integral of $L_{p,\bar{\alpha}}$--supremal convolution for measures similar to Definition \ref{definitionoffconcavity}.
  	
  	\begin{defi}\label{definitionoffconcavity1} Let $p \in [1,\infty)$, $\bar{\alpha}=(\alpha_1,\cdots,\alpha_{n+1})\in [0,\infty]^{n+1}$, and $F \colon \R_+ \to \R$ be an invertible differentiable function.  We say that a measure $\mu$ on $\R^n$ is $F(t)$-concave with respect to the $L_{p,\bar{\alpha}}$--supremal convolution of functions belonging to some class $\mathcal{A}$ of bounded non-negative $\mu$-integrable functions if, for any members $f,g$ belonging to $\mathcal{A}$ and any $t \in [0,1]$, one has that 
  		\begin{equation*}\label{e:functionalMink1stassumption}
  		\int_{\R^n} [((1-t) \times_{p,\bar{\alpha}} f)\oplus_{p,\bar{\alpha}} (t \times_{p,\bar{\alpha}} g)] d\mu \geq F^{-1}\left((1-t)F\left(\int_{\R^n} f d\mu\right) + t F\left(\int_{\R^n} g d\mu\right)\right).
  		\end{equation*}
  	\end{defi}
  	
 Further we pose the following definition of the surface area extension for $L_{p,\bar{\alpha}}$--supremal convolution with measures.
  	
  	\begin{defi} Let $\mu$ be a Borel measure on $\R^n$, $p \geq 1$, and $\bar{\alpha}=(\alpha_1,\cdots,\alpha_{n+1})\in [0,\infty]^{n+1}$. We define the $L_p$-$\mu$-surface area of a $\mu$-integrable function $f \colon \R^n \to \R_+$ with respect to a $\mu$-integrable function $g$ by 
  		\[
  		\mathbb{S}_{\mu,p,\bar{\alpha}}(f,g) := \liminf_{\e \to 0^+} \frac{\int_{\R^n} f \oplus_{p,\bar{\alpha}} (\e \times_{p,\bar{\alpha}} g) d\mu -\int_{\R^n}f d\mu }{\e}.
  		\]
  		If $\mu$ is the Lebesgue measure on $\R^n$, we will simply denote it as $\mathbb{S}_{p,\bar{\alpha}}$.
  	\end{defi}
The  $L_p$-Minkowski's first inequality below for $L_{p,\bar{\alpha}}$ supremal-convolution, which is  generalization of $L_p$-Minkowski's first inequality in terms of $L_{p,s}$ supremal-convolution for functions in \cite{RX}.
  	
  	\begin{theorem}\label{t:functionalMink1st23} Let $p \in[1,\infty)$, $\bar{\alpha}=(\alpha_1,\cdots,\alpha_{n+1})\in [0,\infty]^{n+1}$, and $F \colon \R_+\to \R$ be a differentiable invertible function. Let $\mu$ be a Borel measure on $\R^n$, and assume that $\mu$ is $F(t)$-concave with respect to some class, $\mathcal{A}$, of non-negative bounded $\mu$-integrable functions and the $L_{p,\bar{\alpha}}$--supremal convolution.
  		Then the following inequality holds for any members $f,g$ of the class $\mathcal{A}$: 
  		\begin{equation*}\label{e:functionalMink1stconclusion}
  		\mathbb{S}_{\mu,p,\bar{\alpha}}(f,g) \geq \mathbb{S}_{\mu,p,\bar{\alpha}}(f,f) + \frac{F\left(\int_{\R^n}g d\mu\right) - F\left(\int_{\R^n} f d\mu\right)}{F'\left(\int_{\R^n} f d\mu \right)}.
  		\end{equation*}
  		In particular, when $\int_{\R^n} f d\mu = \int_{\R^n} g d\mu$, we obtain the following isoperimetric type inequality: 
  		\[
  		\mathbb{S}_{\mu,p,\bar{\alpha}}(f,g) \geq \mathbb{S}_{\mu,p,\bar{\alpha}}(f,f). 
  		\]
  	\end{theorem}

  	In particular, we obtain a natural  corollary of Theorem~\ref{t:functionalMink1st} for Lebesgue measure $\mu$.
  	
  	\begin{coro} 
   Let $p \in [1,\infty)$, $\bar{\alpha}=(\alpha_1,\cdots,\alpha_{n+1})\in (0,1]^{n+1}$. Then, for any bounded integrable functions $f,g \colon \R^n \to \R_+$, one has 
  		\[
  		\mathbb{S}_{p,\bar{\alpha}}(f,g) \geq \mathbb{S}_{p,\bar{\alpha}}(f,f) +\frac{\left(\int_{\R^n}g(x)dx \right)^{p\gamma}-\left(\int_{\R^n}f(x)dx \right)^{p\gamma}}{p\gamma\left(\int_{\R^n}f(x)dx\right)^{p\gamma-1}},
  		\]
  			where $\gamma =\left(\sum_{i=1}^{n+1}\alpha_i^{-1}\right)^{-1}.$
  		In particular, when $\int_{\R^n} f(x) dx = \int_{\R^n} g(x) dx>0$, we obtain the following isoperimetric type inequality:
  		\[
  		\mathbb{S}_{p,\bar{\alpha}}(f,g) \geq \mathbb{S}_{p,\bar{\alpha}}(f,f).
  		\]
  	
  	\end{coro}
  	\begin{remark}
  		The definition of $F\circ \mu$ concavity with respect to the $L_{p,\bar{\alpha}}$-supermal-convolution can also be defined accordingly with similar variation formulas exist for functions in (\ref{mixedvolume}) and has similar inequalities in Theorem \ref{t:lpisoperimetricinequalitygeneral}, and thus omitted.
  		\end{remark}

  	\section*{Acknowledgements}
  	This research was partially supported by the Zuckerman STEM Leadership Program.

  \end{document}